\documentclass[a4paper,11pt]{article}
\usepackage[utf8]{inputenc}
\usepackage[T1]{fontenc}
\usepackage[english]{babel}
\usepackage[margin=2.8cm]{geometry}
\usepackage{mathpazo}
\usepackage[scaled=.95]{helvet}
\usepackage{courier}
\usepackage{amsmath,amsfonts,amssymb,amsthm}
\usepackage{graphicx}
\usepackage{natbib,url}
\usepackage{color}
\usepackage[onehalfspacing]{setspace}

\newtheorem{prop}{Proposition}[section]

\newtheorem{rmq}{Remark}[section]  
\newtheorem{theo}{Theorem}[section]  
\newtheorem{lem}{Lemma}[section]

\newtheorem{cor}{Corollary}[section]

\usepackage[draft]{fixme}
\fxusetheme{color}
\FXRegisterAuthor{h}{ah}{H}
\FXRegisterAuthor{p}{ap}{P}
\FXRegisterAuthor{a}{aa}{A}

\newcommand*{\scal}[1]{\left\langle #1 \right\rangle} 

\DeclareMathOperator{\var}{Var}

\begin{document}
\title{Estimating the geometric median in Hilbert spaces with stochastic gradient algorithms: $L^{p}$ and almost sure rates of convergence}
\author{Antoine Godichon \\ Institut de Math\'ematiques de Bourgogne, Universit\'e de Bourgogne, \\
9 Rue Alain Savary, 21078 Dijon, France \\
email: antoine.godichon@u-bourgogne.fr
} 
\maketitle
\begin{abstract}
The geometric median, also called $L^{1}$-median, is often used in robust statistics. Moreover, it is more and more usual to deal with large samples taking values in high dimensional spaces. In this context, a fast recursive estimator has been  introduced by \cite{HC}. This work aims at studying  more precisely the asymptotic behavior of the estimators of the geometric median based on such non linear stochastic gradient algorithms.  The $L^{p}$ rates of convergence as well as almost sure rates of convergence of these estimators are derived in general separable Hilbert spaces. Moreover, the optimal rates of convergence in quadratic mean of the averaged algorithm are also given.
\end{abstract}


\section{Introduction}

The geometric median, also called $L^{1}$-median, is a generalization of the real median introduced by \cite{Hal48}. In the multivariate case, it is closely related to the Fermat-Webber's problem (see \cite{weber1929alfred}), which consists in finding a point minimizing the sum of distances from given points. This is a well known convex optimization problem. The literature is very wide on the estimation of the solution of this problem. One of the most usual method is to use Weiszfeld's algorithm (see \cite{kuhn1973note}), or more recently, to use the algorithm proposed by \cite{beck2013weiszfeld}.

In the more general context of Banach spaces, \cite{Kem87} gives many properties on the median, such as its existence, its uniqueness, and maybe the most important, its robustness. Because of this last property, the median is often used in robust statistics. For example, \cite{minsker2013geometric} considers it in order to get much tighter concentration bounds for aggregation of estimators. \cite{CCM10} propose a recursive algorithm using the median for clustering, which is few sensitive to outliers than the $k$-means. One can also see \cite{chakraborty2014spatial}, \cite{cuevas2014partial}, \cite{BBTW2011} or \cite{Ger08} among others for other examples.

In this context, several estimators of the median are proposed in the literature. In the multivariate case, one of the most usual method is to consider the Fermat-Webber's problem generated by the sample, and to solve it using Weiszfeld's algorithm (see \cite{VZ00} and \cite{MNO2010} for example). This method is fast, but can encounter many difficulties when we deal with a large sample taking values in relatively high dimensional spaces. Indeed, since it requires to store all the data, it can be difficult or impossible to perform the algorithm.

Dealing with high dimensional of functional data is more and more usual. There exists a large recent literature on functional data analysis (see \cite{bongiorno2014contributions}, \cite{ferraty2006nonparametric} or \cite{silverman2005functional} for example), but few of them speak about robustness (see \cite{Cad01} and \cite{cuevas2014partial}).

In this large sample and high dimensional context, recursive algorithms have been introduced by \cite{HC}; a stochastic gradient algorithm, or Robbins-Monro algorithm (see \cite{robbins1951stochastic}, \cite{bartoli}, \cite{Duf97}, \cite{benveniste-book90}, \cite{kushner2003stochastic} among others), and its averaged version (see \cite{PolyakJud92}). It enables us to estimate the median in Hilbert spaces, whose dimension is not necessarily finite, such as functional spaces. The advantage of these algorithms is that they can treat all the data, can be simply updated, and do not require too much computational efforts. Moreover, it has been proven in \cite{HC} that the averaged version and the estimator proposed by \cite{VZ00} have the same asymptotic distribution. Other properties were given, such as the strong consistency of these algorithms. Moreover, the optimal rate of convergence in quadratic mean of the Robbins-monro algorithm as well as non asymptotic confidence balls for both algorithms are given in \cite{CCG2015}.

The aim of this work is to give new asymptotic convergence properties in order to have a deeper knowledge of the asymptotic behaviour of these algorithms. Optimal $L^{p}$ rates of convergence for the Robbins-Monro algorithm are given. This enables, in a first time, to get the optimal rate of convergence in quadratic mean of the averaged algorithm. In a second time, it enables us to get the $L^{p}$ rates of convergence. In a third time, thanks to these results, applying Borel-Cantelli's Lemma, we give an almost sure rate of convergence of the Robbins-Monro algorithm. Finally, applying a law of large numbers for martingales (see \cite{Duf97} for example), we give an almost sure rate of convergence of the averaged algorithm.

The paper is organized as follows. In Section \ref{sectiondefi}, we recall the definition of the median and some important convexity properties. The Robbins-Monro algorithm and its averaged version are defined in Section \ref{sectionalgo}. After recalling the rate of convergence in quadratic mean of the Robbins-Monro algorithm given by \cite{CCG2015}, we give the $L^{p}$-rates of convergence of the stochastic gradient algorithm as well as the optimal rate of convergence in quadratic mean of the averaged algorithm in Section \ref{sectionlp}. Finally, almost sure rates of convergence of the algorithms are given in Section \ref{sectionas}. The lemma that help understanding the structure of the proofs are given all along the text, but the proofs are postponed in an Appendix.

\section{Definitions and convexity properties}\label{sectiondefi}
Let $H$ be a separable Hilbert space, we denote by $\langle .,.  \rangle$ its inner product and by $\| . \| $ the associated norm. Let $X$ be a random variable taking values in $H$, the geometric median $m$ of $X$ is defined by
\begin{equation}
\label{defi}m:= \arg \min_{h\in H} \mathbb{E}\left[ \left\| X-h \right\| - \left\| X \right\| \right] .
\end{equation}
We suppose from now that the following assumptions are fulfilled:
\begin{itemize}
\item[\textbf{(A1)}] $X$ is not concentrated on a straight line: for all $h \in H$, there is $h' \in H$ such that $\left\langle h,h' \right\rangle =0$ and $ \var \left( \left\langle X,h' \right\rangle \right) > 0$.

\item[\textbf{(A2)}] $X$ is not concentrated around single points: there is a positive constant $C$ such that for all $h \in H$, 
\begin{align*}
& \mathbb{E}\left[ \frac{1}{\left\| X-h \right\|}\right] \leq C, & \mathbb{E}\left[ \frac{1}{\left\| X-h \right\|^{2}}\right] \leq C.
\end{align*}
\end{itemize}
Remark that since $\mathbb{E}\left[ \frac{1}{\left\| X-h \right\|^{2}}\right] \leq C$, as a particular case, $\mathbb{E}\left[ \frac{1}{\left\| X - h \right\|}\right] \leq \sqrt{C}$. Note that for the sake of simplicity, even if it means supposing $C \geq 1$, we take $C$ instead of $\sqrt{C}$. Assumption~\textbf{(A1)} ensures that the median $m$ is uniquely defined \citep{Kem87}. Assumption \textbf{(A2)} is not restrictive whenever $d \geq 3$, where $d$ is the dimension of $H$, not necessarily finite (see \cite{HC} and \cite{Chaud92} for more details). Note that many convergence results can be found without Assumption \textbf{(A2)} if we deal with data taking values in compact sets (see \cite{ADPY10} or \cite{yang2010riemannian} for example).

Let $G$ be the function we would like to minimize. It is defined for all $h \in H$ by
\begin{center}
$ G(h) := \mathbb{E}\left[ \left\| X-h \right\| - \left\| X \right\| \right] .$
\end{center} 
This function is convex and many convexity properties are given in \cite{Chaud92}, \cite{Ger08}, \cite{HC} and \cite{CCG2015}. We recall two important ones:
\begin{itemize}
\item[\textbf{(P1)}] $G$ is Fréchet-differentiable and its gradient is given for all $h \in H$ by
\[\Phi (h) := \nabla_{h}G = -\mathbb{E}\left[ \frac{X -h }{\left\| X - h \right\|}\right]. \]
The median $m$ is the unique zero of $\Phi$.
\item[\textbf{(P2)}] $G$ is twice differentiable and for all $h \in H$, $\Gamma_{h}$ stands for the Hessian of $G$ at $h$. Thus, $H$ admits an orthonormal basis composed of eigenvectors of $\Gamma_{h}$, and let $\left( \lambda_{i,h}\right)$ be the eigenvalues of $\Gamma_{h}$, we have
$0 \leq \lambda_{i,h} \leq C. $\\
Moreover, for all positive constant $A$, there is a positive constant $c_{A}$ such that for all $h \in \mathcal{B}\left( 0 , A \right)$, $c_{A} \leq \lambda_{i,h} \leq C$.\\
As a particular case, let $\lambda_{\min}$ be the smallest eigenvalue of $\Gamma_{m}$, there is a positive constant $c_{m}$ such that $ 0< c_{m} < \lambda_{\min} \leq C$.
\end{itemize}

\section{The algorithms}\label{sectionalgo}
Let $X_{1},....,X_{n},...$ be independent random variables with the same law as $X$. We recall the algorithm for estimation of the geometric median introduced by \cite{HC}, defined as follows:
\begin{equation}
\label{algorm} Z_{n+1}  = Z_{n} + \gamma_{n} \frac{X_{n+1} - Z_{n}}{\left\| X_{n+1} - Z_{n} \right\|},
\end{equation}
where the initialization $Z_{1}$ is chosen bounded ($Z_{1}=X_{1}\mathbb{1}_{\left\lbrace\left\| X_{1} \right\| \leq M\right\rbrace}$ for example) or deterministic. The sequence $(\gamma_n)$ of steps is positive and verifies the following usual conditions
\begin{align*}
& \sum_{n\geq 1} \gamma_{n} = \infty , & \sum_{n \geq 1} \gamma_{n}^{2} < \infty . 
\end{align*}
The averaged version of the algorithm (see \cite{PolyakJud92}, \cite{HC}) is given iteratively by
\begin{equation}
\label{algom} \overline{Z}_{n+1} = \overline{Z}_{n} - \frac{1}{n+1}\left( \overline{Z}_{n} - Z_{n+1} \right),
\end{equation}
where $\overline{Z}_{1}=Z_{1}$. This can be written as $\overline{Z}_{n} = \frac{1}{n}\sum_{k=1}^{n} Z_{k}$.
\newline

The algorithm defined by (\ref{algorm}) is a stochastic gradient or Robbins-Monro algorithm. Indeed, it can be written as follows:
\begin{equation}
\label{decxi} Z_{n+1} = Z_{n} - \gamma_{n} \Phi (Z_{n}) + \gamma_{n}\xi_{n+1},
\end{equation}
where $\xi_{n+1}:= \Phi (Z_{n} ) +\frac{X_{n+1} - Z_{n}}{\left\| X_{n+1} - Z_{n} \right\|}$. Let $\mathcal{F}_{n}$ be the $\sigma$-algebra defined by $\mathcal{F}_{n} := \sigma \left( X_{1},...,X_{n} \right) = \sigma \left( Z_{1},...,Z_{n} \right)$. Thus, $\left( \xi_{n}\right)$ is a sequence of martingale differences adapted to the filtration $\left( \mathcal{F}_{n} \right)$. Indeed, for all $n\geq 1$, we have almost surely $\mathbb{E}\left[ \xi_{n+1} | \mathcal{F}_{n} \right] = 0$. Linearizing the gradient,
\begin{equation}
\label{decdelta} Z_{n+1} - m = \left( I_{H} - \gamma_{n}\Gamma_{m} \right) \left( Z_{n} - m \right) + \gamma_{n}\xi_{n+1} - \gamma_{n}\delta_{n} ,
\end{equation}
with $\delta_{n} := \Phi ( Z_{n} ) - \Gamma_{m}\left( Z_{n} - m \right)$. Note that there is a positive deterministic constant $C_{m}$ such that for all $n \geq 1$ (see \cite{CCG2015}), almost surely,
\begin{equation}
\left\| \delta_{n} \right\| \leq C_{m}\left\| Z_{n} - m \right\|^{2} .
\end{equation}
Moreover, since $\Phi (Z_{n}) = \int_{0}^{1}\Gamma_{m+t(Z_{n}-m)}(Z_{n}-m ) dt$, applying convexity property \textbf{(P2)}, one can check that almost surely
\begin{equation}
\left\| \delta_{n} \right\| \leq 2C \left\| Z_{n} - m \right\| .
\end{equation}
\section{$L^{p}$ rates convergence of the algorithms}\label{sectionlp}
We now consider a step sequence $\left( \gamma_{n} \right) $ of the form $\gamma_{n} = c_{\gamma}n^{-\alpha}$ with $c_{\gamma}>0$ and $\alpha \in (1/2 , 1)$. The optimal rate of convergence in quadratic mean of the Robbins-Monro algorithm is given in \cite{CCG2015}. Indeed, it was proven that there are positive constants $c',C'$ such that for all $n \geq 1$,
\begin{equation}
\frac{c'}{n^{\alpha}} \leq \mathbb{E}\left[ \left\| Z_{n} - m \right\|^{2}\right] \leq \frac{C'}{n^{\alpha}} .
\end{equation} 
Moreover, the $L^{p}$ rates of convergence were not given, but it was proven that the $p$-th moments are bounded for all integer $p$: there exists a positive constant $M_{p}$ such that for all $n \geq 1$, 
\begin{equation}\label{majpourproba}
\mathbb{E}\left[ \left\| Z_{n} - m \right\|^{2p}\right] \leq M_{p}.
\end{equation}
\subsection{$L^{p}$ rates of convergence of the Robbins-Monro algorithm}\label{setcionrmvitlp}
\begin{theo}\label{vitconvrm}
Assume \textbf{(A1)} and \textbf{(A2)} hold. For all $p \geq 1$, there is a positive constant $K_{p}$ such that for all $n \geq 1$, 
\begin{equation}\label{inegtheo}
\mathbb{E}\left[ \left\| Z_{n} - m \right\|^{2p}\right] \leq \frac{K_{p}}{n^{p\alpha}} .
\end{equation}
As a corollary, applying Cauchy-Schwarz's inequality, for all $p \geq 1$ and for all $n \geq 1$,
\begin{equation}
\mathbb{E}\left[ \left\| Z_{n}  - m \right\|^{p}\right] \leq \frac{\sqrt{K_{p}}}{n^{\frac{p\alpha}{2}}} .
\end{equation}
\end{theo}
The proof is given in Appendix. Since it was proven (see \cite{CCG2015}) that the rate for $p=1$ is the optimal one, one can check, applying Hölder's inequality, that the given ones for $p \geq 2$ are also optimal. In order to prove this theorem with a strong induction on $p$ and $n$, we have to introduce two technical lemma. The first one gives an upper bound for $\mathbb{E}\left[ \left\| Z_{n+1} - m \right\|^{2p}\right] $ when inequality (\ref{inegtheo}) is verified for all $k \leq p-1$, i.e when the strong induction assumptions are verified. 
\begin{lem}\label{majznp}
Assume \textbf{(A1)} and \textbf{(A2)} hold, let $p \geq 2$, if inequality (\ref{inegtheo}) is verified for all integer for $0$ to $p-1$, there are a rank $n_{\alpha}$ and non-negative constants $c_{0},C_{1},C_{2}$ such that for all $n \geq n_{\alpha}$, 
\begin{equation}
\mathbb{E}\left[ \left\| Z_{n+1} - m \right\|^{2p}\right] \leq \left( 1-c_{0}\gamma_{n} \right) \mathbb{E}\left[ \left\| Z_{n} - m \right\|^{2p}\right] + \frac{C_{1}}{n^{(p+1)\alpha}}+ C_{2}\gamma_{n}\mathbb{E}\left[ \left\| Z_{n} - m \right\|^{2p+2}\right] .
\end{equation}
\end{lem}
The proof is given in Appendix. The following lemma gives an upper bound of $\mathbb{E}\left[ \left\| Z_{n+1} - m \right\|^{2p+2}\right] $ when inequality (\ref{inegtheo}) is verified for all $k \leq p-1$, i.e when the strong induction assumptions are verified. 
\begin{lem}\label{majznpp}
Assume \textbf{(A1)} and \textbf{(A2)} hold, let $p \geq 2$, if inequality (\ref{inegtheo}) is verified for all integer from $0$ to $p-1$, there are a rank $n_{\alpha}$ and non-negative constants $C_{1}',C_{2}'$ such that for all $n \geq n_{\alpha}$,
\begin{equation}
\mathbb{E}\left[ \left\| Z_{n+1} - m \right\|^{2p+2}\right] \leq \left( 1-\frac{2}{n}\right)^{p+1}\mathbb{E}\left[ \left\| Z_{n} - m \right\|^{2p+2}\right] + \frac{C_{1}'}{n^{(p+2)\alpha}} + C_{2}' \gamma_{n}^{2}\mathbb{E}\left[ \left\| Z_{n} - m \right\|^{2p}\right] .
\end{equation}
\end{lem}
The proof is given in Appendix. Note that for the sake of simplicity, we denote by the same way the ranks in Lemma~\ref{majznp} and Lemma~\ref{majznpp}. 
\subsection{Optimal rate of convergence in quadratic mean and $L^{p}$ rates of converge of the averaged algorithm}\label{sectionmoyvtlp}
As done in \cite{HC} and \cite{Pel00}, summing equalities (\ref{decdelta}) and applying Abel's transform, we get
\begin{equation}\label{decmoyennise}
n\Gamma_{m}\left( \overline{Z}_{n} - m \right) = \frac{T_{1}}{\gamma_{1}}- \frac{T_{n+1}}{\gamma_{n}}+ \sum_{k=2}^{n}T_{k}\left( \frac{1}{\gamma_{k}}- \frac{1}{\gamma_{k-1}}\right) + \sum_{k=1}^{n} \delta_{k} + \sum_{k=1}^{n}\xi_{k+1} ,
\end{equation}
with $T_{k}:= Z_{k}-m$. Using this decomposition and Theorem \ref{vitconvrm}, we can derive the $L^{p}$ rates of convergence of the averaged algorithm.
\begin{theo}\label{theol2moy}
Assume \textbf{(A1)} and \textbf{(A2)} hold, for all integer $p \geq 1$, there is a positive constant $A_{p}$ such that for all $n \geq 1$,
\begin{equation}
 \mathbb{E}\left[ \left\| \overline{Z}_{n} - m \right\|^{2p}\right] \leq \frac{A_{p}}{n^{p}}.
\end{equation}
\end{theo}
The proof is given in Appendix. It heavily relies on Theorem \ref{vitconvrm} and on the following lemma which gives a bound of the $p$-th moments of the sum of (non necessarily independent) random variables. Note that this is probably not a new result but we were not able to find a proof in a published reference.
\begin{lem}\label{lemsum}
Let $Y_{1},...,Y_{n}$ be random variables taking values in a normed vector space such that for all positive constant $q$ and for all $k \geq 1$, $\mathbb{E}\left[ \left\| Y_{k} \right\|^{q} \right] < \infty$. Thus, for all constants $a_{1},...,a_{n}$ and for all integer $p$,
\begin{equation}
\mathbb{E}\left[ \left\| \sum_{k=1}^{n} a_{k}Y_{k} \right\|^{p} \right] \leq \left( \sum_{k=1}^{n} \left| a_{k} \right| \left( \mathbb{E}\left[ \left\| Y_{k} \right\|^{p} \right] \right)^{\frac{1}{p}} \right)^{p}
\end{equation}
\end{lem}
The proof is given in Appendix. Finally, the following proposition ensures that the rate of convergence in quadratic mean given by Theorem \ref{theol2moy} is the optimal one.
\begin{prop}\label{propbonvit}
Assume \textbf{(A1)} and \textbf{(A2)} hold, there is a positive constant $c$ such that for all $n \geq 1$,
\[
\mathbb{E}\left[ \left\| \overline{Z}_{n} - m \right\|^{2}\right] \geq \frac{c}{n}.
\]
\end{prop}
Note that applying Hölder's inequality, previous proposition also ensures that the $L^{p}$ rates of convergence given by Theorem \ref{theol2moy} are the optimal ones.

\section{Almost sure rates of convergence}\label{sectionas}
It is proven in \cite{HC} that the Robbins-Monro algorithm converges almost surely to the geometric median. A direct application of Theorem \ref{vitconvrm} and Borel-Cantelli's lemma gives  the following rates of convergence.
\begin{theo}\label{vitas}
Assume \textbf{(A1)} and \textbf{(A2)} hold, for all $\beta < \alpha$, 
\begin{equation}
\left\| Z_{n} - m \right\| = o \left( \frac{1}{n^{\beta /2}}\right) \quad a.s .
\end{equation}
\end{theo}
The proof is given in a Appendix. As a corollary, using decomposition (\ref{decmoyennise}) and Theorem~\ref{vitas}, we get the following bound of the rate of convergence of the averaged algorithm:
\begin{cor}\label{vitasmoy}
Assume \textbf{(A1)} and \textbf{(A2)} hold, for all $\delta >0$,
\begin{equation}
\left\| \overline{Z}_{n} - m \right\| = o \left( \frac{\left(\ln n\right)^{\frac{1+\delta}{2}}}{\sqrt{n}}\right) \quad a.s .
\end{equation}
\end{cor}
The proof is given in Appendix. 

\section*{Acknowledgements}
The author thanks Herv\'e Cardot and Peggy C\'enac for their patience, their trust, and their advice which were very helpful.

\begin{appendix}

\section{Appendix}\label{sectionpreuves}
\subsection{Proofs of Section \ref{setcionrmvitlp}}
First we recall some technical inequalities (see \cite{petrov1995limit} for example).
\begin{lem}\label{lemtechnique}
Let $a,b,c$ be positive constants. Thus,
\begin{align*}
& ab \leq \frac{a^{2}}{2c}+ \frac{b^{2}c}{2} , & a \leq \frac{c}{2} + \frac{a^{2}}{2c} .
\end{align*}
Moreover let $k,p$ be positive integers and $a_{1},...,a_{p}$ be positive constants. Thus,
\[ 
\left( \sum_{j=1}^{p} a_{j} \right)^{k} \leq p^{k-1}\sum_{j=1}^{p}a_{j}^{k}.
\]
\end{lem}

\begin{proof}[Proof of Lemma \ref{majznpp}] We suppose from now that for all $k \leq p-1$, there is a positive constant $K_{k}$ such that for all $n \geq 1$, 
\begin{equation}\label{vitlplem}
\mathbb{E}\left[ \left\| Z_{n} - m \right\|^{2k}\right] \leq \frac{K_{k}}{n^{k\alpha}}.
\end{equation} 
Using decomposition (\ref{decxi}) and Cauchy-Schwarz's inequality, since by definition of $\xi_{n+1}$ we have $\| \xi_{n+1}\| -2\left\langle \Phi (Z_{n}) , \xi_{n+1} \right\rangle  \leq 1$,
\begin{align*}
\left\| Z_{n+1}-m \right\|^{2} & = \left\| Z_{n} - m - \gamma_{n}\Phi (Z_{n}) \right\|^{2} + \gamma_{n}^{2}\left\| \xi_{n+1} \right\|^{2} +2\gamma_{n} \left\langle Z_{n} -m - \gamma_{n} \Phi (Z_{n}) , \xi_{n+1} \right\rangle \\
& \leq \left\| Z_{n} - m - \gamma_{n}\Phi (Z_{n}) \right\|^{2} + \gamma_{n}^{2} + 2\gamma_{n}\left\langle Z_{n}-m , \xi_{n+1}\right\rangle .
\end{align*}
Let $V_{n}:= \left\| Z_{n} - m - \gamma_{n} \Phi (Z_{n}) \right\|^{2}$. Using previous inequality,
\begin{align}
\left\| Z_{n+1}-m \right\|^{2p+2}&   \leq \left( V_{n} + \gamma_{n}^{2} + 2\gamma_{n}\left\langle \xi_{n+1},Z_{n}-m \right\rangle \right)^{p+1} \nonumber \\
 & = \left( V_{n} + \gamma_{n}^{2}\right)^{p+1}  +2(p+1) \gamma_{n} \left\langle \xi_{n+1}, Z_{n}-m \right\rangle \left( V_{n}+\gamma_{n}^{2}\right)^{p} \label{majzn1p1} \\
 & + \sum_{k=2}^{p+1}\binom{p+1}{k}\left( 2\gamma_{n} \left\langle \xi_{n+1},Z_{n}-m \right\rangle \right)^{k}\left( V_{n} + \gamma_{n}^{2}\right)^{p+1-k}. \label{majzn1p1-bis}
\end{align}
We shall upper bound the three terms in (\ref{majzn1p1}) and \eqref{majzn1p1-bis}. Applying Cauchy-Schwarz's inequality and since almost surely $\left\| \Phi (Z_{n}) \right\| \leq C \left\| Z_{n} - m \right\|$, 
\begin{align}
V_{n} & = \left\| Z_{n} - m \right\|^{2} -2 \gamma_{n} \left\langle Z_{n} - m , \Phi (Z_{n} ) \right\rangle + \gamma_{n}^{2}\left\| \Phi (Z_{n} ) \right\|^{2} \nonumber \\
& \leq \left\| Z_{n} - m \right\|^{2} + 2C\gamma_{n} \left\| Z_{n} - m \right\|^{2} + \gamma_{n}^{2}C^2\left\| Z_{n} - m \right\|^{2} \nonumber \\
& \leq \left( 1+c_{\gamma}C \right)^{2} \left\| Z_{n} - m \right\|^{2}.\label{major-V}
\end{align}

We now bound the expectation of the first term in (\ref{majzn1p1}). Indeed,
\begin{align*}
\mathbb{E}  \left[ \left( V_{n}+\gamma_{n}^{2}\right)^{p+1}\right]&  = \mathbb{E}\left[V_{n}^{p+1}\right] + (p+1)\gamma_{n}^{2}\mathbb{E}\left[ V_{n}^{p}\right] + \sum_{k=0}^{p-1}\binom{p+1}{k}\gamma_{n}^{2(p+1-k)}\mathbb{E}\left[ V_{n}^{k}\right] \\
& \leq \mathbb{E}\left[ V_{n}^{p+1}\right] + (p+1)\left( 1+c_{\gamma}C\right)^{2p}\gamma_{n}^{2}\mathbb{E}\left[ \left\| Z_{n}-m \right\|^{2p}\right] \\
& + \sum_{k=0}^{p-1}\binom{p+1}{k}\left( 1+c_{\gamma}C\right)^{2k}\gamma_{n}^{2(p+1-k)}\mathbb{E}\left[ \left\| Z_{n}-m \right\|^{2k}\right] .
\end{align*}
Applying inequality (\ref{vitlplem}), 
\begin{align*}
\sum_{k=0}^{p-1}  \binom{p+1}{k}\left( 1+c_{\gamma}C\right)^{2k}\gamma_{n}^{2(p+1-k)}\mathbb{E}\left[ \left\| Z_{n}-m \right\|^{2k}\right]  & \leq \sum_{k=0}^{p-1}\binom{p+1}{k}\left( 1+c_{\gamma}C\right)^{2k}\gamma_{n}^{2(p+1-k)} \frac{K_{k}}{n^{k\alpha}} \\
& \leq \sum_{k=0}^{p-1}\binom{p+1}{k}\frac{\left( 1+c_{\gamma}C\right)^{2k}c_{\gamma}^{2(p+1-k)}K_{k}}{n^{(2p+2-k)\alpha}} \\
& = O \left( \frac{1}{n^{(p+3)\alpha}}\right) .
\end{align*}
As a conclusion, there is a non-negative constant $A_{1}$ such that for all $n \geq 1$,
\begin{equation}\label{eq*}
\mathbb{E}\left[ \left( V_{n} + \gamma_{n}^{2}\right)^{p+1}\right] \leq \mathbb{E}\left[ V_{n}^{p+1}\right] + (p+1)\left( 1+c_{\gamma}C\right)^{2p}\gamma_{n}^{2}\mathbb{E}\left[ \left\| Z_{n}-m \right\|^{2p}\right] + \frac{A_{1}}{n^{(p+3)\alpha}}.
\end{equation}
We now bound the second term in (\ref{majzn1p1}). Using the facts that $\left( \xi_{n+1}\right)$ is a sequence of martingale differences adapted to the filtration $\left( \mathcal{F}_{n}\right)$ and that $Z_{n}$ is $\mathcal{F}_{n}$-measurable, 
\begin{equation}\label{eq**}
\mathbb{E}\left[ 2(p+1) \gamma_{n} \left\langle \xi_{n+1}, Z_{n}-m \right\rangle \left( V_{n}+\gamma_{n}^{2}\right)^{p} \right] =0 .
\end{equation}
Finally, we bound the last term in (\ref{majzn1p1-bis}), denoted by $(*)$. Since almost surely $\left\| \xi_{n} \right\| \leq 2$, applying Cauchy-Schwarz's inequality, 
\begin{align}\label{majotrespourrie}
\notag (*) & \leq  \sum_{k=2}^{p+1}\sum_{j=0}^{p+1-k}\binom{p+1}{k}\binom{p+1-k}{j}2^{k}\gamma_{n}^{2j+k}\left\| \xi_{n+1} \right\|^{k}\left\| Z_{n} - m \right\|^{k}V_{n}^{p+1-k-j} \\
& \leq   \sum_{k=2}^{p+1}\sum_{j=0}^{p+1-k}\binom{p+1}{k}\binom{p+1-k}{j}2^{2k}\gamma_{n}^{2j+k}\left\| Z_{n} - m \right\|^{k}V_{n}^{p+1-k-j}
\end{align}
Since almost surely $V_{n} \leq \left( 1+c_{\gamma}C\right)^{2}\left\| Z_{n}-m \right\|^{2}$ (see inequality \eqref{major-V}),
\begin{align}
\notag (*) & \leq \sum_{k=2}^{p+1}\sum_{j=0}^{p+1-k}\binom{p+1}{k}\binom{p+1-k}{j}2^{2k}\gamma_{n}^{2j+k}\left( 1+c_{\gamma}C\right)^{2p+2-2k-2j}\left\| Z_{n}-m \right\|^{2p+2-k-2j} \\
\label{eqmaj*}& = \sum_{k=2}^{p+1}\sum_{j=1}^{p+1-k}\binom{p+1}{k}\binom{p+1-k}{j}2^{2k}\gamma_{n}^{2j+k}\left( 1+c_{\gamma}C\right)^{2p+2-2k-2j}\left\| Z_{n}-m \right\|^{2p+2-k-2j} \\
\notag & + \sum_{k=3}^{p+1}\binom{p+1}{k}2^{2k}\gamma_{n}^{k}\left( 1+c_{\gamma}C\right)^{2p+2-2k-2j}\left\| Z_{n}-m\right\|^{2p+2-k} + 16\binom{p+1}{2}\gamma_{n}^{2}\left( 1+c_{\gamma}^{2}C^{2}\right)^{2p} \left\| Z_{n}-m \right\|^{2p} .
\end{align}
We bound the expectation of the two first terms on the right-hand side of (\ref{eqmaj*}). For the first one, applying Cauchy-Schwarz's inequality, 
\begin{align*}
 \mathbb{E} &   \left[ \sum_{k=2}^{p+1}\sum_{j=1}^{p+1-k}\binom{p+1}{k}\binom{p+1-k}{j}2^{2k}\gamma_{n}^{2j+k}\left( 1+c_{\gamma}C\right)^{2p+2-2k-2j}\left\| Z_{n}-m \right\|^{2p+2-k-2j} \right] \\
& \leq \sum_{k=2}^{p+1}\sum_{j=1}^{p+1-k}\binom{p+1}{k}\binom{p+1-k}{j} \\
 & 2^{2k}\gamma_{n}^{2j+k}\left(1+c_{\gamma}C\right)^{2p+2-2k-2j}\sqrt{\mathbb{E}\left[ \left\| Z_{n} - m \right\|^{2(p-j)}\right] \mathbb{E}\left[ \left\| Z_{n} - m \right\|^{2(p-k-j+2)}\right]}.
\end{align*}
Applying inequality (\ref{vitlplem}), 
\begin{align*}
\mathbb{E} & \left[ \sum_{k=2}^{p+1}\sum_{j=1}^{p+1-k}\binom{p+1}{k}\binom{p+1-k}{j}2^{2k}\gamma_{n}^{2j+k}\left(1+c_{\gamma}C\right)^{2p+2-2k-2j}\left\| Z_{n}-m \right\|^{2p+2-k-2j} \right] \\
& \leq \sum_{k=2}^{p+1}\sum_{j=1}^{p+1-k}\binom{p+1}{k}\binom{p+1-k}{j}2^{2k}\gamma_{n}^{2j+k}\left(1+c_{\gamma}C\right)^{2p+2-2k-2j}\frac{\sqrt{K_{p-j}}}{n^{\frac{p-j}{2}\alpha}}\frac{\sqrt{K_{p-k-j+2}}}{n^{\frac{p-k-j+2}{2}\alpha}} \\
& = o \left( \frac{1}{n^{(p+2)\alpha}}\right).
\end{align*}
Similarly, for the second term on the right-hand side of (\ref{eqmaj*}), applying Cauchy-Schwarz's inequality, let
\begin{align*}
(**) & :=  \sum_{k=3}^{p+1}  \binom{p+1}{k}2^{2k}\gamma_{n}^{k}\left( 1+c_{\gamma}C\right)^{2p+2-2k}\mathbb{E}\left[\left\| Z_{n}-m\right\|^{2p+2-k} \right] \\
 &  \leq  \sum_{k=4}^{p+1}\binom{p+1}{k}2^{2k}\gamma_{n}^{k}\left( 1+c_{\gamma}C\right)^{2p+2-2k}\mathbb{E}\left[\left\| Z_{n}-m\right\|^{2p+2-k} \right] \\.
 & + 64\binom{p+1}{3}\left( 1+c_{\gamma}C\right)^{2p-4}\mathbb{E}\left[ \left\| Z_{n}-m \right\|^{2p-1}\right] \\
& \leq \sum_{k=4}^{p+1}\binom{p+1}{k}2^{2k}\gamma_{n}^{k}\left( 1+c_{\gamma}C \right)^{2p+2-2k} \sqrt{\mathbb{E}\left[\left\| Z_{n}-m\right\|^{2(p+3-k)} \right]\mathbb{E}\left[ \left\| Z_{n} - m \right\|^{2(p-1)}\right]}\\
& + 64\binom{p+1}{3}\left( 1+c_{\gamma}C\right)^{2p-4}\mathbb{E}\left[ \left\| Z_{n}-m \right\|^{2p-1}\right] .
\end{align*}
Applying Lemma~\ref{lemtechnique} and inequality (\ref{vitlplem})
\begin{align*}
(**) & \leq \sum_{k=4}^{p+1}\binom{p+1}{k}2^{2k}\gamma_{n}^{k}\left( 1+c_{\gamma}C\right)^{2p+2-2k}\frac{\sqrt{ K_{p+3-k}K_{p-1}}}{n^{(p+1-k/2)\alpha}} \\
  &  + 32\binom{p+1}{3}\left( 1+c_{\gamma}C\right)^{2p-4}\gamma_{n}^{3}\left( \mathbb{E}\left[ \left\| Z_{n}-m \right\|^{2p}\right] + \mathbb{E}\left[\left\| Z_{n}-m \right\|^{2p-2}\right]\right) \\
  & = O \left( \frac{1}{n^{(p+2)\alpha}}\right) + 32\binom{p+1}{3}\left( 1+c_{\gamma}C\right)^{2p-4}\gamma_{n}^{3} \mathbb{E}\left[ \left\| Z_{n}-m \right\|^{2p}\right] .
\end{align*}
Finally, let us denote by $(***)$ the expectation of the term in (\ref{majzn1p1-bis}), there is a positive constant $A_{2}$ such that for all $n \geq 1$,
\begin{equation}\label{eq***}
(***) \leq \frac{A_{2}}{n^{(p+2)\alpha}} + 16 \binom{p+1}{2}\left( 1+c_{\gamma}C\right)\gamma_{n}^{2} \mathbb{E}\left[ \left\| Z_{n}-m \right\|^{2p}\right] + 32\binom{p+1}{3}\left( 1+c_{\gamma}C\right)^{2p-4}\gamma_{n}^{3} \mathbb{E}\left[ \left\| Z_{n}-m \right\|^{2p}\right].
\end{equation}
Applying inequalities (\ref{eq*}),(\ref{eq**}) and (\ref{eq***}), there are positive constants $C_{1}'',C_{2}'$ such that for all $n \geq 1$,
\begin{equation}\label{eq1}
\mathbb{E}\left[ \left\| Z_{n+1}-m \right\|^{2p+2}\right] \leq \mathbb{E}\left[ V_{n}^{p+1}\right] + \frac{C_{1}''}{n^{(p+2)\alpha}}+ C_{2}'\gamma_{n}^{2}\mathbb{E}\left[ \left\| Z_{n}-m \right\|^{2p}\right] .
\end{equation}
In order to conclude, we need to bound $\mathbb{E}\left[ V_{n}^{p+1}\right] $. Applying Lemma 5.2 in \cite{CCG2015}, there are a positive constant $c$ and a rank $n_{\alpha}$ such that for all $n \geq n_{\alpha}$,
\begin{equation}\label{eq2}
\mathbb{E}\left[ V_{n}^{p+1}\mathbb{1}_{\left\lbrace\left\| Z_{n}-m \right\| \leq cn^{1-\alpha}\right\rbrace} \right] \leq \left( 1 - \frac{2}{n}\right)^{p+1}\mathbb{E}\left[ \left\| Z_{n}-m \right\|^{2p+2}\right] .
\end{equation}
Finally, since there is a positive constant $c_{0}$ such that almost surely $\left\| Z_{n} - m \right\| \leq c_{0}n^{1-\alpha}$ and since almost surely $V_{n} \leq \left( 1+c_{\gamma}C\right)^{2}\left\| Z_{n}-m \right\|^{2}$, 
\begin{align*}
\mathbb{E}\left[ V_{n}^{p+1} \mathbb{1}_{\left\lbrace\left\| Z_{n} - m \right\| \geq cn^{1-\alpha}\right\rbrace}\right] & \leq \left( 1+c_{\gamma}C\right)^{2p+2}\mathbb{E}\left[ \left\| Z_{n} - m \right\|^{2p+2}\mathbb{1}_{\left\lbrace\left\| Z_{n} - m \right\| \geq cn^{1-\alpha}\right\rbrace}\right] \\
& \leq \left( 1+c_{\gamma}C \right)^{2p+2}c_{0}^{2p+2}n^{(2p+2)(1-\alpha)}\mathbb{E}\left[ \mathbb{1}_{\left\lbrace\left\| Z_{n} - m \right\| \geq cn^{1-\alpha}\right\rbrace} \right] \\
& = \left( 1+c_{\gamma}C\right)^{2p+2}c_{0}^{2p+2}n^{(2p+2)(1-\alpha)}\mathbb{P}\left( \left\| Z_{n} - m \right\| \geq cn^{1-\alpha}\right).
\end{align*}
Applying inequality (\ref{majpourproba}) and Markov's inequality,
\begin{align*}
\mathbb{E}\left[ V_{n}^{p+1} \mathbb{1}_{\left\lbrace\left\| Z_{n} - m \right\| \geq cn^{1-\alpha}\right\rbrace}\right] & \leq \left( 1+c_{\gamma}C\right)^{2p+2}c_{0}^{2p+2}n^{(2p+2)(1-\alpha)} \frac{\mathbb{E}\left[ \left\| Z_{n} - m \right\|^{2q} \right]}{(cn)^{2q(1-\alpha)}} \\
& \leq \frac{\left( 1+c_{\gamma}C\right)^{2p+2}c_{0}^{2p+2}n^{(2p+2)(1-\alpha)}}{c^{2q(1-\alpha)}} \frac{M_{q}}{n^{2q(1-\alpha)}} \\
& = O \left( \frac{1}{n^{2q(1-\alpha) - (2p+2)(1-\alpha)}}\right).
\end{align*}
Taking $q \geq p+1 + \frac{(p+2) \alpha}{2(1-\alpha)}$,
\begin{equation}\label{eq3}
\mathbb{E}\left[ V_{n}^{p+1}\mathbb{1}_{\left\lbrace\left\| Z_{n} - m \right\| \geq cn^{1-\alpha}\right\rbrace}\right] = O \left( \frac{1}{n^{(p+2)\alpha}}\right) .
\end{equation}
Finally, using inequalities (\ref{eq1}) to (\ref{eq3}), there is a positive constant $C_{1}'$ such that for all $n \geq n_{\alpha}$,
\begin{equation}
\mathbb{E}\left[ \left\| Z_{n+1} - m \right\|^{2p+2} \right] \leq \left( 1-\frac{2}{n}\right)^{p+1}\mathbb{E}\left[ \left\| Z_{n}-m \right\|^{2p+2}\right] + \frac{C_{1}'}{n^{(p+2)\alpha}} + C_{2}'\gamma_{n}^{2}\mathbb{E}\left[ \left\| Z_{n} - m \right\|^{2p}\right] .
\end{equation}

\end{proof}
\begin{proof}[Proof of Lemma \ref{majznp}] Since the eigenvalues of $\Gamma_{m}$ belong to $[ \lambda_{\min} , C]$, there are a rank $n_{\alpha}$ and a positive constant $c'$ such that for all $n \geq n_{\alpha}$, we have $\left\| I_{H} - \gamma_{n}\Gamma_{m} \right\|_{op} \leq  1-\lambda_{\min}\gamma_{n} $ and $ 0 \leq \left(1-\lambda_{\min}\gamma_{n}\right)^{2} + 4C^{2}\gamma_{n}^{2} \leq 1-c'\gamma_{n}$. Using decomposition (\ref{decdelta}) and Cauchy-Schwarz's inequality, since $\left\| \delta_{n} \right\| \leq 2C \left\| Z_{n}-m \right\| $ and $\left\| \xi_{n+1} \right\|^{2} - 2\left\langle \Phi (Z_{n} ) , \xi_{n+1} \right\rangle \leq 1$, we have for all $n \geq n_{\alpha}$,
\begin{align}\label{eq1'}
\notag \left\| Z_{n+1}-m \right\|^{2} & \leq \left( 1- \lambda_{\min}\gamma_{n}\right)^{2}\left\| Z_{n}-m \right\|^{2} +2\gamma_{n} \left\langle \xi_{n+1}, Z_{n}-m-\gamma_{n}\Phi (Z_{n}) \right\rangle \\
\notag & - 2\gamma_{n}\left\langle \left( I_{H} - \gamma_{n}\Gamma_{m}\right) \left( Z_{n}-m \right) , \delta_{n} \right\rangle + \gamma_{n}^{2}\left\| \delta_{n} \right\|^{2} + \gamma_{n}^{2}\left\| \xi_{n+1} \right\|^{2} \\
& \leq \left( 1-c'\gamma_{n}\right)\left\| Z_{n}-m \right\|^{2} + 2\gamma_{n}\left\| Z_{n}-m \right\| \left\| \delta_{n} \right\| + \gamma_{n}^{2} + 2\gamma_{n}\left\langle Z_{n}-m , \xi_{n+1} \right\rangle .
\end{align}
Thus, for all integers $p \geq 1$ and $n \geq n_{\alpha}$,
\begin{align}\label{equation+++}
\notag \mathbb{E}\left[ \left\| Z_{n+1} - m \right\|^{2p} \right] & \leq \left( 1-c'\gamma_{n} \right) \mathbb{E}\left[ \left\| Z_{n} - m \right\|^{2}\left\| Z_{n+1}-m \right\|^{2p-2}\right] + 2\gamma_{n}\mathbb{E}\left[ \left\| Z_{n} - m \right\| \left\| \delta_{n} \right\| \left\| Z_{n+1}- m \right\|^{2p-2}\right] \\
& +\gamma_{n}^{2}\mathbb{E}\left[ \left\| Z_{n+1} - m \right\|^{2p-2}\right] +2\gamma_{n} \mathbb{E}\left[ \left\langle Z_{n} - m , \xi_{n+1} \right\rangle \left\| Z_{n+1}-m \right\|^{2p-2}\right] .
\end{align}
In order to bound each term in previous inequality, we give a new upper bound of $\left\| Z_{n+1}-m \right\|^{2p-2}$. By convexity of $G$, we have almost surely $V_n \leq \left\| Z_{n}-m \right\|^{2} + \gamma_{n}^{2}$, and inequality (\ref{majzn1p1}) can be written as
\begin{align*}
\left\| Z_{n+1}-m\right\|^{2p-2} & \leq \left( \left\| Z_{n}-m \right\|^{2} + \gamma_{n}^{2}\right)^{p-1}  +2(p-1) \gamma_{n} \left\langle \xi_{n+1}, Z_{n}-m \right\rangle \left( \left\| Z_{n}-m \right\|^{2}+\gamma_{n}^{2}\right)^{p-2} \\
  & + \sum_{k=2}^{p-1}\binom{p-1}{k}\left| 2\gamma_{n} \left\langle \xi_{n+1},Z_{n}-m \right\rangle \right|^{k}\left( \left\| Z_{n}-m  \right\|^{2} + \gamma_{n}^{2}\right)^{p-1-k}. 
\end{align*}
Applying Cauchy-Schwarz's inequality, since $\| \xi_{n+1} \| \leq 2$, 
\begin{align}\label{eq2'}
\left\| Z_{n+1}-m\right\|^{2p-2} & \leq \left( \left\| Z_{n}-m \right\|^{2} + \gamma_{n}^{2}\right)^{p-1}  +2(p-1) \gamma_{n} \scal{\xi_{n+1},Z_{n}-m} \left( \left\| Z_{n}-m \right\|^{2}+\gamma_{n}^{2}\right)^{p-2} \\
\notag  & + \sum_{k=2}^{p-1}\binom{p-1}{k}2^{2k}\gamma_{n}^{k}\left\| Z_{n}-m \right\|^{k}\left( \left\| Z_{n}-m  \right\|^{2} + \gamma_{n}^{2}\right)^{p-1-k}. 
\end{align}
Note that if $p \leq 2$, the last term on the right-hand side of previous inequality is equal to $0$. Applying previous inequality, we can now bound each term in inequality (\ref{equation+++}). 
\newline

\textbf{Step 1: Bounding $\left( 1-c'\gamma_{n} \right)\mathbb{E}\left[ \left\| Z_{n} - m \right\|^{2}\left\| Z_{n+1}-m \right\|^{2p-2} \right]$.} 

We will bound each term which appears when we multiply $ \left( 1-c'\gamma_{n} \right)\left\| Z_{n} - m \right\|^{2}$ by the bound given by inequality (\ref{eq2'}). First, applying inequalities (\ref{vitlplem}), 
\begin{align*}
\mathbb{E} & \left[ \left( 1-c'\gamma_{n} \right) \left\| Z_{n}-m \right\|^{2}\left( \left\| Z_{n} - m \right\|^{2} + \gamma_{n}^{2}\right)^{p-1}\right] \\
& = \left( 1-c'\gamma_{n}\right) \mathbb{E}\left[ \left\| Z_{n}-m \right\|^{2p}\right] + \sum_{k=0}^{p-2}\binom{p-1}{k}\left( 1-c'\gamma_{n}\right) \gamma_{n}^{2(p-1-k)}\mathbb{E}\left[ \left\| Z_{n}-m \right\|^{2k+2}\right] \\
& \leq \left( 1-c'\gamma_{n}\right) \mathbb{E}\left[ \left\| Z_{n}-m \right\|^{2p}\right] + \sum_{k=0}^{p-2}\binom{p-1}{k}\left( 1-c'\gamma_{n}\right)c_{\gamma}^{2(p-1-k)}\frac{K_{k+1}}{n^{(2p-1-k)\alpha}}.
\end{align*}
Since for all $k \leq p-2$, we have $2p-1-k \geq p+1$, there is a positive constant $B_{1}$ such that for all $n \geq n_{\alpha}$,
\begin{equation}
\mathbb{E}  \left[ \left( 1-c'\gamma_{n} \right) \left\| Z_{n}-m \right\|^{2}\left( \left\| Z_{n} - m \right\|^{2} + \gamma_{n}^{2}\right)^{p-1}\right] \leq \left( 1-c'\gamma_{n}\right) \mathbb{E}\left[ \left\| Z_{n}-m \right\|^{2p}\right] + \frac{B_{1}}{n^{(p+1)\alpha}} .
\end{equation}
Moreover, using the facts that $\left( \xi_{n}\right)$ is a martingale differences sequence adapted to the filtration $\left( \mathcal{F}_{n}\right)$, and that $Z_{n}$ is $\mathcal{F}_{n}$-measurable, 
\begin{equation}
\mathbb{E}\left[ \left( 1-c'\gamma_{n}\right)\left\| Z_{n}-m \right\|^{2}2\left( p-1 \right) \gamma_{n} \left\langle \xi_{n+1}, Z_{n}-m \right\rangle \left( \left\| Z_{n}-m \right\|^{2} + \gamma_{n}^{2}\right)^{p-2}\right] = 0 .
\end{equation}
We can now suppose that $p \geq 3$, since otherwise the last term in inequality (\ref{eq2'}) is equal to $0$. Let
\begin{align*}
(\star )  & := \left( 1-c'\gamma_{n} \right) \mathbb{E}\left[ \left\| Z_{n}-m \right\|^{2} \sum_{k=2}^{p-1}\binom{p-1}{k}2^{2k}\gamma_{n}^{k}\left\| Z_{n}-m \right\|^{k} \left( \left\| Z_{n} - m \right\|^{2} + \gamma_{n}^{2}\right)^{p-1-k}\right] \\
& \leq \left( 1-c'\gamma_{n}\right)\sum_{k=2}^{p-1}\binom{p-1}{k}2^{p-2+k}\gamma_{n}^{k}\left( \mathbb{E}\left[ \left\| Z_{n} - m \right\|^{2p -k}\right] + \gamma_{n}^{2(p-1-k)}\mathbb{E}\left[ \left\| Z_{n}-m \right\|^{k+2} \right] \right) .
\end{align*}
Applying Cauchy-Schwarz's inequality,
\begin{align}
\notag  (\star ) & \leq \left( 1-c'\gamma_{n}\right)\sum_{k=2}^{p-1}\binom{p-1}{k}2^{p-2+k}\gamma_{n}^{k} \\
\notag & \left( \sqrt{\mathbb{E}\left[ \left\| Z_{n}-m \right\|^{2(p-1)}\right]\mathbb{E}\left[ \left\| Z_{n} - m \right\|^{2(p+1-k)}\right]} + \gamma_{n}^{2(p-1-k)}\sqrt{\mathbb{E}\left[ \left\| Z_{n} - m \right\|^{2k}\right] \mathbb{E}\left[ \left\| Z_{n} - m \right\|^{4}\right]} \right)
\end{align}
Finally, applying inequality (\ref{vitlplem}),
\begin{align}\label{calculpourri}
\notag (\star)  & \leq \left( 1-c'\gamma_{n}\right)\sum_{k=2}^{p-1}\binom{p-1}{k}2^{p-2+k}\gamma_{n}^{k}\left( \frac{\sqrt{K_{p-1}K_{p+1-k}}}{n^{(p-k/2)\alpha}} + \gamma_{n}^{2(p-1-k)}\frac{\sqrt{K_{k}K_{2}}}{n^{\frac{(k +2)\alpha}{2}}} \right) \\
& = O \left( \frac{1}{n^{(p+1)\alpha}}\right),
\end{align}
because for all $2 \leq k \leq p-1$ and $p \geq 3$, we have $ p +k/2 \geq p+1$ and $ 2p-\frac{1}{2}k -1 \geq p+1$. Thus, there is a positive constant $B_{1}'$ such that for all $n \geq n_{\alpha}$,
\begin{equation}\label{eq1''}
\mathbb{E}\left[ \left( 1-c'\gamma_{n}\right)\left\| Z_{n}-m \right\|^{2}\left\| Z_{n+1}-m \right\|^{2p-2}\right]  \leq \left( 1-c'\gamma_{n}\right) \mathbb{E}\left[ \left\| Z_{n} -m \right\|^{2p}\right] + \frac{B_{1}'}{n^{(p+1)\alpha}}.
\end{equation}

\textbf{Step 2: Bounding $2\gamma_{n}\mathbb{E}\left[ \left\langle \xi_{n+1},Z_{n} - m \right\rangle \left\| Z_{n+1} - m \right\|^{2p-2}\right]$.}

Applying the fact that $\left( \xi_{n}\right) $ is a martingale differences sequence adapted to the filtration $\left( \mathcal{F}_{n} \right)$ and applying inequality (\ref{eq2'}), let
\begin{align*}
(\star \star ) & :=  \mathbb{E}\left[ 2\gamma_{n}\left\langle \xi_{n+1} , Z_{n}-m \right\rangle \left\| Z_{n+1}-m \right\|^{2p-2} \right] \\
 & \leq 4(p-1)\gamma_{n}^{2}\mathbb{E}\left[ \left\langle \xi_{n+1}, Z_{n}-m \right\rangle^{2}\left( \left\| Z_{n}-m \right\|^{2}+\gamma_{n}^{2}\right)^{p-2}\right] \\
 \end{align*}
Since $\| \xi_{n+1}\| \leq 2$ and applying Cauchy-Schwarz's inequality, 
 \begin{align*}
(\star \star ) & \leq 4(p-1) \gamma_{n}^{2}\mathbb{E}\left[ \left( \left\| \xi_{n+1} \right\| \left\| Z_{n} - m \right\| \right)^{2}\left( \left\| Z_{n} - m \right\|^{2} +  \gamma_{n}^{2} \right)^{p-2}\right] \\
& \leq 16(p-1) \gamma_{n}^{2}\mathbb{E}\left[ \left\| Z_{n}-m \right\|^{2}\left( \left\| Z_{n}-m \right\|^{2}+\gamma_{n}^{2}\right)^{p-2}\right] .
\end{align*}
With the help of Lemma \ref{lemtechnique},
\[
(\star \star )  \leq 2^{p+2}(p-1) \gamma_{n}^{2}\left( \mathbb{E}\left[ \left\| Z_{n} - m \right\|^{2(p-1)}\right] + \gamma_{n}^{2(p-2)}\mathbb{E}\left[ \left\| Z_{n}-m \right\|^{2}\right] \right) ,
\]
Applying previous inequality and inequality (\ref{vitlplem}), there is a positive constant $B_{2}'$ such that
\begin{equation}
\label{eq2''} \mathbb{E}\left[ 2\gamma_{n}\left\langle \xi_{n+1} , Z_{n}-m \right\rangle \left\| Z_{n+1}-m \right\|^{2p-2} \right] \leq \frac{B_{2}'}{n^{(p+1)\alpha}} .
\end{equation}

\textbf{Step 3: Bounding $\gamma_{n}^{2}\mathbb{E}\left[ \left\| Z_{n+1}-m \right\|^{2p-2}\right]$.} 

Applying inequality (\ref{vitlplem}), 
\begin{align}
\label{eq3''}  \notag \gamma_{n}^{2}\mathbb{E}\left[ \left\| Z_{n+1}-m \right\|^{2p-2}\right] & \leq \gamma_{n}^{2} \frac{K_{p-1}}{(n+1)^{p-1}} \\
& = O \left( \frac{1}{n^{(p+1)\alpha}}\right) .
\end{align}

\textbf{Step 4: Bounding $2\gamma_{n}\mathbb{E}\left[ \left\| Z_{n} - m \right\| \left\| \delta_{n} \right\| \left\| Z_{n+1}-m \right\|^{2p-2}\right].$} 

As in step 1, we will bound each term which appears when we multiply $ 2\gamma_{n}\left\| Z_{n} - m \right\| \left\| \delta_{n} \right\|$ by the bound given by inequality (\ref{eq2'}). Since almost surely $\left\| \delta_{n} \right\| \leq 2C \left\| Z_{n}-m \right\|$, applying inequality (\ref{calculpourri}), one can check
\begin{align*}
2\gamma_{n} & \mathbb{E}\left[ \left\| Z_{n}-m \right\| \left\| \delta_{n} \right\| \sum_{k=2}^{p-1}\binom{p-1}{k}2^{2k}\gamma_{n}^{k}\left\| Z_{n}-m \right\|^{k} \left( \left\| Z_{n} - m \right\|^{2} + \gamma_{n}^{2}\right)^{p-1-k}\right] \\
&  \leq 4C\gamma_{n}   \mathbb{E}\left[ \left\| Z_{n}-m \right\|^{2} \sum_{k=2}^{p-1}\binom{p-1}{k}2^{2k}\gamma_{n}^{k}\left\| Z_{n}-m \right\|^{k} \left( \left\| Z_{n} - m \right\|^{2} + \gamma_{n}^{2}\right)^{p-1-k}\right] \\
& = o \left( \frac{1}{n^{(p+1)\alpha}}\right) .
\end{align*}
Moreover, since $\left( \xi_{n}\right)$ is a martingale differences sequence adapted to the filtration $\left( \mathcal{F}_{n} \right)$, 
\begin{equation}
\notag \mathbb{E}\left[ 2\gamma_{n} \left\| Z_{n}-m \right\| \left\| \delta_{n} \right\| 2\left( p-1 \right) \gamma_{n} \left\langle \xi_{n+1}, Z_{n}-m \right\rangle \left( \left\| Z_{n}-m \right\|^{2} + \gamma_{n}^{2}\right)^{p-2}\right] =0 .
\end{equation}
Finally, since almost surely  $\left\| \delta_{n} \right\| \leq C_{m}\left\| Z_{n} - m \right\|^{2}$ and $\left\| \delta_{n} \right\| \leq 2C \left\| Z_{n} - m \right\|$, applying Lemma~\ref{lemtechnique}, 
\begin{align*}
 (\star \star \star )  & := \mathbb{E}  \left[ 2\gamma_{n}\left\| Z_{n}-m \right\| \left\| \delta_{n} \right\|\left( \left\| Z_{n} - m \right\|^{2} + \gamma_{n}^{2}\right)^{p-1}\right] \\
 & \leq 2^{p-1}\gamma_{n}\mathbb{E}\left[ \left\| Z_{n} - m \right\|^{2p-1}\left\| \delta_{n} \right\| \right] + 2^{p-1}\gamma_{n}^{2p-1}\mathbb{E}\left[ \left\| Z_{n} - m \right\| \left\| \delta_{n} \right\| \right] \\
 & \leq 2^{p-1}C_{m}\gamma_{n} \mathbb{E}\left[ \left\| Z_{n}-m \right\|^{2p+1}\right]+  2^{p-1}C\gamma_{n}^{2p-1}\mathbb{E}\left[ \left\| Z_{n}-m \right\|^{2} \right] .
\end{align*}
Applying Lemma~\ref{lemtechnique},
\[
(\star \star \star)  \leq \frac{1}{2}c'\gamma_{n}\mathbb{E}\left[ \left\| Z_{n}-m \right\|^{2p}\right] + 2^{2p-2}\frac{C_{m}^{2}}{c'}\gamma_{n}\mathbb{E}\left[ \left\| Z_{n}-m \right\|^{2p+2}\right]   + O \left( \frac{1}{n^{(p+1)\alpha}}\right) .
\]
Thus, there are positive constants $B_{3}',B_{4}'$ such that
\begin{equation}
\label{eq4''} 2\mathbb{E}\left[ \left\| Z_{n}-m \right\| \left\| \delta_{n} \right\| \left\| Z_{n+1}-m \right\|^{2p-2}\right] \leq \frac{1}{2}c'\gamma_{n}\mathbb{E}\left[ \left\| Z_{n} - m \right\|^{2p}\right] + B_{3}'\gamma_{n}\mathbb{E}\left[ \left\| Z_{n} - m \right\|^{2p+2}\right] + \frac{B_{4}'}{n^{(p+1)\alpha}} .
\end{equation}

\textbf{Step 5: Conclusion.} Taking $c_{0} = \frac{1}{2}c'$, applying inequalities (\ref{eq1''}),(\ref{eq2''}),(\ref{eq3''}) and (\ref{eq4''}), there are positive constants $C_{1},C_{2}$ such that for all $n \geq n_{\alpha}$,
\begin{equation}\notag
\mathbb{E}\left[ \left\| Z_{n+1} - m \right\|^{2p}\right] \leq \left( 1-c_{0}\gamma_{n} \right) \mathbb{E}\left[ \left\| Z_{n} - m \right\|^{2p}\right] + \frac{C_{1}}{n^{(p+1)\alpha}}+ C_{2}\gamma_{n}\mathbb{E}\left[ \left\| Z_{n} - m \right\|^{2p+2}\right] .
\end{equation}

\end{proof}
\begin{proof}[Proof of Theorem \ref{vitconvrm}] We prove with the help of a complete induction that for all $p \geq 1$, and for all $\beta \in (\alpha , \frac{p+2}{p}\alpha - \frac{1}{p})$, there are positive constants $K_{p},C_{\beta ,p}$ such that for all $n \geq 1$,
\begin{align*}
& \mathbb{E}\left[ \left\| Z_{n} - m \right\|^{2p}\right] \leq \frac{K_{p}}{n^{p\alpha}} ,& \mathbb{E}\left[ \left\| Z_{n} - m \right\|^{2p+2}\right] \leq \frac{C_{\beta , p}}{n^{\beta p}} .
\end{align*}
This result is proven in \cite{CCG2015} for $p=1$. Let $p \geq 2$ and let us suppose from now that for all integer $k \leq p-1$, there are positive constant $K_{k}$ such that for all $n \geq 1$,
\begin{equation}
\mathbb{E}\left[ \left\| Z_{n} - m \right\|^{2k}\right] \leq \frac{K_{k}}{n^{k\alpha}}.
\end{equation}
We now split the end of the proof into two steps.

\medskip
\textbf{Step 1: Calibration of the constants.}

In order to simplify the demonstration thereafter, we introduce some constants and notations. Let $\beta$ be a constant such that $\frac{p+2}{p}\alpha - \frac{1}{p} >\beta > \alpha$ and let $K_{p}', K_{p,\beta}'$ be constants such that $K_{p}' \geq 2^{1+p\alpha} C_{1}c_{0}^{-1}c_{\gamma}^{-1}$, ($C_{1}$ is defined in Lemma \ref{majznp}),  and $2K_{p}' \geq K_{p,\beta}' \geq K_{p}' \geq 1$. By definition of $\beta $, there is a rank $n_{p,\beta} \geq n_{\alpha}$ ($n_{\alpha}$ is defined in Lemma \ref{majznp} and in Lemma \ref{majznpp}) such that for all $n \geq n_{p,\beta}$,
\begin{align}\label{defnbeta}
\notag \left( 1-c_{0}\gamma_{n} \right)\left( \frac{n+1}{n}\right)^{p\alpha} + \frac{1}{2}c_{0}\gamma_{n} + \frac{2^{\alpha + \beta p +1}c_{\gamma}C_{2}}{(n+1)^{\alpha +(\beta - \alpha )p}} \leq 1 , \\
\left( 1- \frac{2}{n}\right)^{p+1} \left( \frac{n+1}{n}\right)^{p \beta} + \left( C_{1}' + C_{2}'c_{\gamma}^{2}\right) 2^{(p+2)\alpha}\frac{1}{(n+1)^{(p+2) \alpha - p \beta }} \leq 1 ,
\end{align}
with $C_{2}$ defined in Lemma \ref{majznp} and $C_{1}',C_{2}'$  are defined in Lemma \ref{majznpp}. Because $\beta > \alpha$,
\begin{align*}
\left( 1-c_{0}\gamma_{n} \right)\left( \frac{n+1}{n}\right)^{p\alpha} + \frac{1}{2}c_{0}\gamma_{n} + \frac{2^{\alpha + \beta p +1}c_{\gamma}C_{2}}{(n+1)^{\alpha +(\beta - \alpha )p}} & = 1-c_{0}\gamma_{n} + o \left( \frac{1}{n}\right) + \frac{1}{2}c_{0}\gamma_{n} + O \left( \frac{1}{n^{\alpha + (\beta - \alpha )p}}\right) \\
& = 1 - \frac{1}{2}c_{0}\gamma_{n} + o \left( \frac{1}{n^{\alpha}}\right).  
\end{align*}
In the same way, since $\beta < \frac{p+2}{p}\alpha - \frac{1}{p}$, $p\beta  < 2p+2$ and
\begin{align*}
\left( 1- \frac{2}{n}\right)^{p+1} \left( \frac{n+1}{n}\right)^{p \beta} &  + \left( C_{1}' + C_{2}'c_{\gamma}^{2}\right) 2^{(p+2)\alpha}\frac{1}{(n+1)^{(p+2) \alpha - p \beta }}  = 1- \left( 2p+2 -p\beta \right)\frac{1}{n} + o \left( \frac{1}{n}\right) .
\end{align*}

\medskip

\textbf{Step 2: The induction.}

Let us take $K_{p}' \geq n_{p,\beta}^{p\alpha}\mathbb{E}\left[ \left\| Z_{n_{p,\beta}} - m \right\|^{2p}\right]$ and  $K_{p,\beta}' \geq n_{p,\beta}^{p\alpha}\mathbb{E}\left[ \left\| Z_{n_{p,\beta}} - m \right\|^{2p+2}\right]$, we will prove by induction that for all $n \geq n_{p,\beta}$, 
\begin{align*}
& \mathbb{E}\left[ \left\| Z_{n} - m \right\|^{2p}\right]\leq \frac{K_{p}'}{n^{p\alpha}}, & \mathbb{E}\left[ \left\| Z_{n} - m \right\|^{2p+2}\right] \leq \frac{K_{p,\beta}'}{n^{p\beta}} .
\end{align*}
Applying Lemma \ref{majznp} and by induction, since $2K_{p}' \geq K_{p, \beta}' \geq K_{p}' \geq 1$,
\begin{align*}
\mathbb{E}\left[ \left\| Z_{n+1} - m \right\|^{2p}\right] & \leq \left( 1-c_{0}\gamma_{n}\right)\mathbb{E}\left[ \left\| Z_{n} - m \right\|^{2p}\right] + \frac{C_{1}}{n^{(p+1)\alpha}} + C_{2}\gamma_{n}\mathbb{E}\left[ \left\| Z_{n} - m \right\|^{2p+2}\right] \\
& \leq \left( 1-c_{0}\gamma_{n}\right) \frac{K_{p}'}{n^{p\alpha}} + \frac{C_{1}}{n^{(p+1)\alpha}} + C_{2}\gamma_{n}\frac{K_{p,\beta}'}{n^{p\beta}} \\
& \leq \left( 1-c_{0}\gamma_{n} \right) \frac{K_{p}'}{n^{p\alpha}} + \frac{C_{1}}{n^{(p+1)\alpha}} + 2C_{2}\gamma_{n} \frac{K_{p}'}{n^{p\beta}}. 
\end{align*}
Factorizing by $\frac{K_{p}'}{(n+1)^{p \alpha }}$, 
\begin{align*}
\mathbb{E}\left[ \left\| Z_{n+1} - m \right\|^{2p}\right] & \leq \left( 1-c_{0}\gamma_{n} \right) \left( \frac{n+1}{n}\right)^{p\alpha}\frac{K_{p}'}{(n+1)^{p\alpha}} + \left( \frac{n+1}{n}\right)^{p\alpha}C_{1}\frac{1}{(n+1)^{p\alpha}n^{\alpha}} \\
&  + 2c_{\gamma}C_{2}\left( \frac{n+1}{n}\right)^{\alpha + \beta p}\frac{K_{p}'}{(n+1)^{\beta p + \alpha}} \\
& \leq \left( 1-c_{0}\gamma_{n} \right) \left( \frac{n+1}{n}\right)^{p\alpha}\frac{K_{p}'}{(n+1)^{p\alpha}} + \frac{2^{p\alpha}C_{1}c_{\gamma}^{-1}\gamma_{n}}{(n+1)^{p\alpha}} +\frac{2^{\alpha +\beta p +1}c_{\gamma}C_{2}}{(n+1)^{\alpha +(\beta - \alpha )p}}\frac{K_{p}'}{(n+1)^{p\alpha }}.
\end{align*}
Since $K_{p}' \geq 2^{1+ p\alpha}C_{1}c_{\gamma}^{-1}c_{0}^{-1}$,
\begin{align*}
\mathbb{E}\left[ \left\| Z_{n+1}-m \right\|^{2p}\right] & \leq \left( 1-c_{0}\gamma_{n}\right)\left( \frac{n+1}{n}\right)^{p\alpha}\frac{K_{p}'}{(n+1)^{p\alpha}} + \frac{1}{2}\gamma_{n}c_{0}\frac{K_{p}'}{(n+1)^{p\alpha}} + \frac{2^{\alpha + \beta p +1}c_{\gamma}C_{2}}{(n+1)^{\alpha + (\beta - \alpha)p}}\frac{K_{p}'}{(n+1)^{p\alpha}} \\
& \leq \left( \left( 1-c_{0}\gamma_{n} \right)\left( \frac{n+1}{n}\right)^{p\alpha} + \frac{1}{2}c_{0}\gamma_{n} + \frac{2^{\alpha + \beta p +1}c_{\gamma}C_{2}}{(n+1)^{\alpha +(\beta - \alpha )p}}\right)\frac{K_{p}'}{(n+1)^{p\alpha}}.
\end{align*}
By definition of $n_{p,\beta}$ (see (\ref{defnbeta})), 
\begin{equation}
\mathbb{E}\left[ \left\| Z_{n+1}-m \right\|^{2p}\right] \leq \frac{K_{p}'}{(n+1)^{p\alpha}} .
\end{equation}
In the same way, applying Lemma \ref{majznpp} and by induction, since $K_{p,\beta}' \geq K_{p}' \geq 1$, for all $n~\geq~n_{p,\beta}$,
\begin{align*}
\mathbb{E}\left[ \left\| Z_{n+1}-m \right\|^{2p+2}\right] & \leq \left( 1- \frac{2}{n}\right)^{p+1}\mathbb{E}\left[ \left\| Z_{n} - m \right\|^{2p+2}\right] +\frac{C'_1}{n^{(p+2)\alpha}} + C_{2}'\gamma_{n}^{2}\mathbb{E}\left[ \left\| Z_{n} - m \right\|^{2p}\right] \\
& \leq \left( 1-\frac{2}{n}\right)^{p+1}\frac{K_{p,\beta}'}{n^{p\beta}} + \frac{C'_{1}}{n^{(p+2)\alpha}} + C_{2}'\gamma_{n}^{2}\frac{K_{p}'}{n^{p\alpha}} \\
& \leq \left( 1- \frac{2}{n}\right)^{p+1}\frac{K_{p,\beta}'}{n^{p\beta}} + \frac{C_{1}'K_{p,\beta}'}{n^{(p+2)\alpha}} + C_{2}'\gamma_{n}^{2}\frac{K_{p,\beta}'}{n^{p\alpha}}.
\end{align*}
Factorizing by $\frac{K_{p,\beta}'}{(n+1)^{p\beta}}$,
\begin{align*}
\mathbb{E}\left[ \left\| Z_{n+1}-m \right\|^{2p+2}\right] & \leq \left( 1- \frac{2}{n}\right)^{p+1}\left( \frac{n+1}{n}\right)^{p\beta}\frac{K_{p,\beta}'}{(n+1)^{p\beta}} + C_{1}'\left( \frac{n+1}{n}\right)^{(p+2)\alpha}\frac{1}{(n+1)^{(p+2)\alpha -p\beta}}\frac{K_{p,\beta}'}{(n+1)^{p\beta}}\\
&  + C_{2}'c_{\gamma}^{2}\left( \frac{n+1}{n}\right)^{(p+2)\alpha}\frac{1}{(n+1)^{(p+2)\alpha -p\beta}}\frac{K_{p,\beta}'}{(n+1)^{p\beta}} \\
& \leq \left( \left( 1-\frac{2}{n}\right)^{p+1}\left( \frac{n+1}{n}\right)^{p\beta} + 2^{(p+2)\alpha}\frac{C_{1}' + C_{2}'c_{\gamma}^{2}}{(n+1)^{(p+2)\alpha - p\beta}}\right) \frac{K_{p,\beta}'}{(n+1)^{p\beta}}  .
\end{align*}
By definition of $n_{p,\beta}$,
\begin{equation}
\mathbb{E}\left[ \left\| Z_{n+1}-m \right\|^{2p+2}\right] \leq \frac{K_{p,\beta}'}{(n+1)^{p\beta}},
\end{equation}
which concludes the induction. In order to conclude the proof, we just have to take $K_{p} ~ \geq~ K_{p}',K_{p , \beta }  ~ \geq~ K_{p, \beta}'$, and
\begin{align*}
  & K_{p} \geq \max_{k < n_{p, \beta}}n^{k\alpha}\mathbb{E}\left[ \left\| Z_{k} - m \right\|^{2p}\right] , & K_{p,\beta} \geq \max_{k < n_{p, \beta}}n^{k\beta}\mathbb{E}\left[ \left\| Z_{k} - m \right\|^{2p+2}\right] .
\end{align*}
\end{proof}

\subsection{Proofs of Section \ref{sectionmoyvtlp}}

\begin{proof}[Proof of Lemma \ref{lemsum}]
For all integers $p \geq 1$ and $n \geq 1$, there are positive constants $c_{b}$, $ b \in \mathbb{N}^{n}$, such that for all non-negative real numbers $y_{k}$, $k=1,...,n$, 
\begin{equation}
\left( \sum_{k=1}^{n} y_{k} \right)^{p} = \sum_{ b=(b_{1},...,b_{n}) \in \mathbb{N}^{n} , b_{1}+b_{2}+...+b_{n}=p}c_{b}y_{1}^{b_{1}}...y_{n}^{b_{n}}.
\end{equation}
As a particular case, applying a classical generalization of Hölder's inequality (see \cite{smarandache1996collected}, page 179, for example) ,
\begin{align*}
\mathbb{E}\left[ \left\| \sum_{k=1}^{n} a_{k} Y_{k} \right\|^{p} \right] & \leq \mathbb{E}\left[ \left( \sum_{k=1}^{n} \left| a_{k} \right| \left\| Y_{k} \right\| \right)^{p} \right] \\
& = \sum_{ b=(b_{1},...,b_{n}) \in \mathbb{N}^{n},  b_{1}+b_{2}+...+b_{n}=p}c_{b}\left| a_{1}\right|^{b_{1}}...\left| a_{n}\right|^{b_{n}}\mathbb{E}\left[ \left\| Y_{1}\right\|^{b_{1}} ... \left\| Y_{n} \right\|^{b_{n}} \right] \\
& \leq \sum_{ b=(b_{1},...,b_{n}) \in \mathbb{N}^{n} , b_{1}+b_{2}+...+b_{n}=p}c_{b} \left| a_{1} \right|^{b_{1}} ... \left| a_{n} \right|^{b_{n}} \left(\mathbb{E}\left[ \left\| Y_{1}\right\|^{p}\right] \right)^{\frac{b_{1}}{p}}...\left(\mathbb{E}\left[ \left\| Y_{n}\right\|^{p}\right] \right)^{\frac{b_{n}}{p}} \\
& = \sum_{ b=(b_{1},...,b_{n}) \in \mathbb{N}^{n} , b_{1}+b_{2}+...+b_{n}=p}c_{b} \left( \left| a_{1}\right| \left( \mathbb{E}\left[ \left\| Y_{1} \right\|^{p} \right] \right)^{\frac{1}{p}}\right)^{b_{1}}...\left( \left| a_{n}\right| \left( \mathbb{E}\left[ \left\| Y_{n} \right\|^{p} \right] \right)^{\frac{1}{p}}\right)^{b_{n}} \\
& = \left( \sum_{k=1}^{n} \left| a_{k} \right| \left( \mathbb{E}\left[ \left\| Y_{k} \right\|^{p}\right]\right)^{\frac{1}{p}} \right)^{p} . 
\end{align*}
\end{proof}

The following lemma give the $L^{p}$ rates of convergence of the martingale term. Note that this is probably not a new result, but we were not able to find a proof in a published reference.
\begin{lem}\label{lemmart}
Let $\left( \xi_{n} \right)$ be a sequence of martingale differences taking values in a Hilbert space $H$ adapted to a filtration $\left( \mathcal{F}_{n} \right)$. Suppose that there is a non-negative constant $M$ such that for all $n \geq 1$, $\left\| \xi_{n} \right\| \leq M $ almost surely. Then, for all integer $p \geq 1$, there is a positive constant $C_{p}$ such that for all $n \geq 1$,
\[ 
\mathbb{E}\left[ \left\| \sum_{k=2}^{n} \xi_{k} \right\|^{2p} \right] \leq C_{p}n^{p}.
\]
\end{lem}

\begin{proof}[Proof of Lemma A.2]
We prove Lemma A.2 with the help of a strong induction on $p\geq 1$. First, if $p=1$, since $\left( \xi_{n} \right)$ is a sequence of martingale adapted to a filtration $\left( \mathcal{F}_{n} \right)$, 
\begin{align*}
\mathbb{E}\left[ \left\| \sum_{k=2}^{n} \xi_{k} \right\|^{2} \right] &  = \sum_{k=2}^{n} \mathbb{E}\left[ \left\| \xi_{k}\right\|^{2} \right] + 2 \sum_{k=2}^{n}\sum_{k'=k}^{n} \mathbb{E}\left[ \left\langle \xi_{k} , \xi_{k'} \right\rangle \right] \\
& \leq (n-1)M^{2} + 2 \sum_{k=2}^{n} \sum_{k'=k}^{n} \mathbb{E}\left[ \left\langle \xi_{k}, \mathbb{E}\left[ \xi_{k'}| \mathcal{F}_{k'-1} \right] \right\rangle \right] \\
& = (n-1)M^{2}.
\end{align*}
Let $p \geq 2$ and for all $n \geq 2$, $M_{n} := \sum_{k=2}^{n}\xi_{k}$. We suppose from now that for all $k \leq p-1$, there is a positive constant $C_{k}$ such that for all $n \geq 2$,
\[
\mathbb{E}\left[ \left\| M_{n} \right\|^{2k} \right] \leq C_{k}(n-1)^{k} .
\]
For all $n \geq 2$,
\begin{align*}
\left\| M_{n+1} \right\|^{2} & = \left\| M_{n} \right\|^{2} + 2 \left\langle M_{n} , \xi_{n+1} \right\rangle + \left\| \xi_{n+1} \right\|^{2} \\
& \leq \left\| M_{n} \right\|^{2} + 2 \left\langle M_{n} , \xi_{n+1} \right\rangle + M^{2} .
\end{align*}
Thus,
\begin{align}\label{ineqmart1}
\notag \left\| M_{n+1}\right\|^{2p} & \leq \left( \left\| M_{n} \right\|^{2} + M^{2}\right)^{p} + 2 \left\langle M_{n} , \xi_{n+1} \right\rangle \left( \left\| M_{n} \right\|^{2} + M^{2} \right)^{p-1} \\
& + \sum_{k=2}^{p}\binom{p}{k}\left| 2 \left\langle M_{n} , \xi_{n+1} \right\rangle \right|^{k} \left( \left\| M_{n} \right\|^{2} + M^{2} \right)^{p-k}. 
\end{align}
We now bound the expectation of the three terms on the right-hand side of previous inequality. First, by induction,
\begin{align*}
\mathbb{E}\left[ \left( \left\| M_{n} \right\|^{2} + M^{2} \right)^{p}\right] & = \mathbb{E}\left[ \left\| M_{n} \right\|^{2p}\right] + \sum_{k=1}^{p}\binom{p}{k}M^{2k}\mathbb{E}\left[ \left\| M_{n} \right\|^{2p-2k}\right] \\
& \leq \mathbb{E}\left[ \left\| M_{n} \right\|^{2p} \right] + \sum_{k=1}^{p}\binom{p}{k}M^{2k}C_{p-k}n^{p-k} \\
& \leq \mathbb{E}\left[ \left\| M_{n} \right\|^{2p} \right] + O \left( n^{p-1} \right) .
\end{align*}
Moreover, since $\left( \xi_{n} \right)$ is a sequence of martingale differences adapted to a filtration $\left( \mathcal{F}_{n} \right)$, and since $M_{n}$ is $\mathcal{F}_{n}$-measurable,
\begin{equation}
\mathbb{E}\left[ \left\langle M_{n} , \xi_{n+1} \right\rangle \left( \left\| M_{n} \right\|^{2} + M^{2} \right)^{p-1} \right] = \mathbb{E}\left[ \left\langle M_{n} , \mathbb{E}\left[ \xi_{n+1} |\mathcal{F}_{n} \right] \right\rangle \left( \left\| M_{n} \right\|^{2} + M^{2} \right)^{p-1} \right] = 0 .
\end{equation}
Finally, applying Cauchy-Schwarz's inequality and Lemma A.1, since $\left\| \xi_{n} \right\| \leq M$, let
\begin{align*}
(*)& := \sum_{k=2}^{p}\binom{p}{k} \mathbb{E}\left[ \left| 2 \left\langle M_{n} , \xi_{n+1} \right\rangle \right|^{k} \left( \left\| M_{n} \right\|^{2} + M^{2}\right)^{p-k}\right] \\
& \leq \sum_{k=2}^{p}\binom{p}{k}2^{k}M^{k}\mathbb{E}\left[ \left\| M_{n} \right\|^{k} \left( \left\| M_{n} \right\|^{2} + M^{2}\right)^{p-k} \right] \\
& \leq \sum_{k=2}^{p}\binom{p}{k} 2^{p-1}M^{k} \left( \mathbb{E}\left[ \left\| M_{n} \right\|^{2p-k} \right] + M^{2p-2k}\mathbb{E}\left[ \left\| M_{n} \right\|^{k} \right] \right) .
\end{align*}
Applying Cauchy-Schwarz's inequality and by induction,
\begin{align}\label{ineqmart2}
\notag (*) & \leq \sum_{k=2}^{p} \binom{p}{k}2^{p-1}M^{k} \left( \sqrt{\mathbb{E}\left[ \left\| M_{n} \right\|^{2p-2} \right] \mathbb{E}\left[ \left\| M_{n} \right\|^{2(p+1-k)}\right]} + M^{2p-2k}\sqrt{\mathbb{E}\left[ \left\| M_{n} \right\|^{2}\right] \mathbb{E}\left[ \left\| M_{n} \right\|^{2k-2} \right]} \right) \\
\notag & \leq \sum_{k=2}^{p}\binom{p}{k}2^{p-1}M^{k} \left( \sqrt{C_{p-1}C_{p+1-k}}n^{p-k/2} + M^{2p-2k}\sqrt{C_{1}C_{k-1}}n^{k/2}\right) \\
& = O \left( n^{p-1} \right) ,
\end{align}
since $p \geq 2$. Thus, thanks to inequalities (\ref{ineqmart1}) to (\ref{ineqmart2}), there is a non-negative constant $A_{p}$ such that for all $n \geq 1$,
\begin{align*}
\mathbb{E}\left[ \left\| M_{n+1} \right\|^{2p} \right] & \leq \mathbb{E	}\left[ \left\| M_{n} \right\|^{2p} \right] + A_{p}n^{p-1} \\
& \leq \left\| \xi_{2} \right\|^{2p} + A_{p}\sum_{k=2}^{n} k^{p-1} \\
& \leq M^{2p} + A_{p}n^{p},
\end{align*} 
which concludes the induction and the proof.
\end{proof}

\begin{proof}[Proof of Theorem 4.2] Let us recall the following decomposition 
\begin{equation}
\label{decmoy} n\Gamma_{m}\left( \overline{Z}_{n} - m \right) = \frac{T_{1}}{\gamma_{1}} - \frac{T_{n+1}}{\gamma_{n}} + \sum_{k=2}^{n}T_{k}\left( \frac{1}{\gamma_{k}}- \frac{1}{\gamma_{k-1}} \right) + \sum_{k=1}^{n} \delta_{k} + \sum_{k=1}^{n} \xi_{k+1},
\end{equation}
with $T_{n} = Z_{n} - m$. Let $\lambda_{\min}>0$ be the smallest eigenvalue of $\Gamma_{m}$, we have with Lemma~A.1,
\begin{align*}
\mathbb{E}\left[ \left\| \overline{Z}_{n} - m \right\|^{2p}\right] & \leq \frac{5^{2p-1}}{\lambda_{\min}^{2p}n^{2p}} \mathbb{E}\left[ \left\| \frac{T_{1}}{\gamma_{1}}\right\|^{2p}\right] + \frac{5^{2p-1}}{\lambda_{\min}^{2p}n^{2p}}\mathbb{E}\left[ \left\| \frac{T_{n+1}}{\gamma_{n}}\right\|^{2p}\right] + \frac{5^{2p-1}}{\lambda_{\min}^{2p}n^{2p}} \mathbb{E}\left[ \left\| \sum_{k=2}^{n} T_{k}\left( \frac{1}{\gamma_{k}} - \frac{1}{\gamma_{k-1}}\right) \right\|^{2p}\right] \\
& + \frac{5^{2p-1}}{\lambda_{\min}^{2p }n^{2p}} \mathbb{E}\left[ \left\| \sum_{k=1}^{n} \delta_{k} \right\|^{2p}\right] + \frac{5^{2p-1}}{\lambda_{\min}^{2p}n^{2p}} \mathbb{E}\left[ \left\| \sum_{k=1}^{n} \xi_{k+1} \right\|^{2p}\right] .
\end{align*}
We now bound each term at the right-hand side of previous inequality. Since $Z_{1}$ is almost surely bounded, we have $\frac{1}{n^{2p}}\mathbb{E}\left[ \left\|  \frac{T_{1}}{\gamma_{1}}\right\|^{2p} \right] = O \left( \frac{1}{n^{2p}}\right)$. Moreover, with Theorem 4.1, 
\begin{align}\label{eqeq1}
\notag \frac{1}{n^{2p}}\mathbb{E}\left[ \frac{\left\| T_{n+1}\right\|^{2p}}{\gamma_{n}^{2p}}\right] & \leq  \frac{1}{c_{\gamma}^{2p}}\frac{1}{n^{2p -2p\alpha}}\frac{K_{1}}{(n+1)^{p\alpha}} \\
& = o \left( \frac{1}{n^{p}}\right) ,
\end{align}
since $\alpha < 1$. In the same way, since $\left| \frac{1}{\gamma_{k-1}}- \frac{1}{\gamma_{k}}\right| \leq 2\alpha c_{\alpha}^{-1}k^{\alpha -1}$, applying Lemma 4.3 and Theorem 4.1,
\begin{align}\label{eqeq2}
\notag \frac{1}{n^{2p}}\mathbb{E}\left[ \left\| \sum_{k=2}^{n} T_{k}\left( \frac{1}{\gamma_{k}} - \frac{1}{\gamma_{k-1}} \right) \right\|^{2p}\right] & \leq \frac{1}{n^{2p}}\left( \sum_{k=2}^{n} \left( \frac{1}{\gamma_{k}} - \frac{1}{\gamma_{k-1}} \right) \left(\mathbb{E}\left[ \left\| T_{k}\right\|^{2p}\right]\right)^{\frac{1}{2p}} \right)^{2p} \\
\notag & \leq \frac{1}{n^{2p}}\left( \sum_{k=2}^{n} 2\alpha c_{\gamma}^{-1}k^{\alpha -1} \left(\frac{K_{p}}{k^{\alpha p}}\right)^{\frac{1}{2p}} \right)^{2p} \\
\notag & \leq \frac{2^{2p} \alpha^{2p}c_{\gamma}^{-2p}K_{p}}{n^{2p}}\left( \sum_{k=2}^{n}\frac{1}{k^{1-\alpha /2}} \right)^{2p} \\
& = O \left( \frac{1}{n^{(2-\alpha ) p}}\right) .
\end{align}
Finally, since $\left\| \delta_{n} \right\| \leq C_{m}\left\| Z_{n} - m \right\|^{2}$, applying Lemma 4.3 and Theorem 4.1,
\begin{align}\label{eqeq3}
\notag \frac{1}{n^{2p}}\mathbb{E}\left[ \left\| \sum_{k=1}^{n} \delta_{k} \right\|^{2p}\right] & \leq \frac{1}{n^{2p}} \left( \sum_{k=1}^{n}\left( \mathbb{E}\left[ \left\| \delta_{k} \right\|^{2p}\right]\right)^{\frac{1}{2p}} \right)^{2p} \\
\notag &  \leq \frac{C_{m}^{2p}}{n^{2p}}\left( \sum_{k=1}^{n} \left( \mathbb{E}\left[ \left\| Z_{k}-m \right\|^{4p}\right]\right)^{\frac{1}{2p}} \right)^{2p} \\
\notag & \leq \frac{C_{m}^{2p}K_{2p}}{n^{2p}}\left( \sum_{k=1}^{n} \frac{1}{k^{\alpha}} \right)^{2p} \\
& = O \left( \frac{1}{n^{2\alpha p}} \right) .
\end{align}
Since $\alpha > 1/2$, we have $\frac{1}{n^{2p}}\mathbb{E}\left[ \left\| \sum_{k=1}^{n} \delta_{k} \right\|^{2p}\right] = o \left( \frac{1}{n^{p}}\right) $ . Finally, applying Lemma A.2, there is a positive constant $C_{p}$ such that for all $n \geq 1$,
\begin{align}
\notag  \frac{1}{n^{2p}}\mathbb{E}\left[ \left\| \sum_{k=1}^{n} \xi_{k+1} \right\|^{2p}\right] & \leq \frac{1}{n^{2p}}C_{p}(n+1)^{p} \\
\label{eqeq4} &  = O \left( \frac{1}{n^{p}}\right) .
\end{align}
We deduce from inequalities (\ref{eqeq1}) to (\ref{eqeq4}), that for all integer $p \geq 1$, there is a positive constant $A_{p}$ such that for all $n \geq 1$,
\begin{equation}
\mathbb{E}\left[ \left\| \overline{Z}_{n} -m \right\|^{2p}\right] \leq \frac{A_{p}}{n^{p}}.
\end{equation}
\end{proof}

\begin{proof}[Proof of Proposition 4.1]
We now give a lower bound of $\mathbb{E}\left[ \left\| \overline{Z}_{n} - m \right\|^{2}\right]$. One can check that $\frac{1}{n}\sum_{k=1}^{n} \xi_{k+1}$ is the dominant term in decomposition (\ref{decmoy}). Indeed, decomposition (\ref{decmoy}) can be written as
\begin{equation}
\Gamma_{m} \left( \overline{Z}_{n} - m \right) = \frac{1}{n}\sum_{k=1}^{n} \xi_{k+1} + \frac{1}{n}R_{n},
\end{equation}
with
\[
R_{n} := \frac{T_{1}}{\gamma_{1}} - \frac{T_{n+1}}{\gamma_{n}} + \sum_{k=2}^{n} T_{k} \left( \frac{1}{\gamma_{k}} - \frac{1}{\gamma_{k-1}}\right) + \sum_{k=1}^{n} \delta_{k} .
\]
Applying inequalities (\ref{eqeq1}), (\ref{eqeq2}) and (\ref{eqeq3}), one can check that
\begin{equation}
\frac{1}{n^{2}}\mathbb{E}\left[ \left\| R_{n} \right\|^{2}\right] = o \left( \frac{1}{n}\right) .
\end{equation}
Moreover,
\begin{equation}
\mathbb{E}\left[ \left\| \Gamma_{m}\left( \overline{Z}_{n} - m \right) \right\|^{2} \right] = \frac{1}{n^{2}}\mathbb{E}\left[ \left\| \sum_{k=1}^{n} \xi_{k+1} \right\|^{2} \right] + \frac{1}{n^{2}}\mathbb{E}\left[ \left\| R_{n} \right\|^{2}  \right] + \frac{2}{n^{2}} \mathbb{E}\left[ \left\langle \sum_{k=1}^{n} \xi_{k+1} , R_{n} \right\rangle \right] 
\end{equation}
Applying Cauchy-Schwarz's inequality and Lemma 4.3, there is a positive constant $C_{1}$ such that for all $n \geq 1$, 
\begin{align*}
\frac{2}{n^{2}}\mathbb{E} \left[ \left| \left\langle \sum_{k=1}^{n} \xi_{k+1} , R_{n} \right\rangle \right| \right] & \leq 2\mathbb{E}\left[ \frac{1}{n^{2}}\left\| \sum_{k=1}^{n} \xi_{k+1} \right\| \left\| R_{n} \right\| \right] \\
& \leq 2 \sqrt{ \frac{1}{n^{2}}\mathbb{E}\left[  \left\| \sum_{k=1}^{n} \xi_{k+1} \right\|^{2}\right]} \sqrt{\frac{1}{n^{2}}\mathbb{E}\left[ \left\| R_{n} \right\|^{2} \right]} \\
& \leq \frac{2\sqrt{C_{1}}}{\sqrt{n}} \sqrt{\frac{1}{n^{2}}\mathbb{E}\left[ \left\| R_{n} \right\|^{2} \right]} \\
& = o \left( \frac{1}{n}\right) .
\end{align*}

Moreover, since $\mathbb{E}\left[ \left\| \xi_{n+1}\right\|^{2}\right] = 1- \mathbb{E}\left[ \left\| \Phi (Z_{n}) \right\|^{2} \right] $ (see \cite{HC} for details), using the fact that $\left( \xi_{n} \right)$ is a sequence of martingale differences adapted to the filtration $\left(\mathcal{F}_{n} \right)$, we get
\begin{align*}
 \frac{1}{n^{2}}\mathbb{E}\left[ \left\| \sum_{k=1}^{n} \xi_{k+1} \right\|^{2}\right] & = \frac{1}{n^{2}}\sum_{k=1}^{n} \mathbb{E}\left[ \left\| \xi_{k+1} \right\|^{2}\right] + 2 \sum_{k=1}^{n} \sum_{k'=k+1}^{n} \mathbb{E}\left[ \left\langle \xi_{k+1}, \xi_{k'+1} \right\rangle \right]  \\
 & = \frac{1}{n} - \frac{1}{n^{2}} \sum_{k=1}^{n} \mathbb{E}\left[ \left\| \Phi (Z_{k} ) \right\|^{2}\right] + 2 \sum_{k=1}^{n} \sum_{k'=k+1}^{n}\mathbb{E}\left[ \left\langle \xi_{k+1} , \mathbb{E}\left[ \xi_{k'+1} |\mathcal{F}_{k'} \right] \right\rangle \right] \\  
& = \frac{1}{n} - \frac{1}{n^{2}} \sum_{k=1}^{n} \mathbb{E}\left[ \left\| \Phi (Z_{k} ) \right\|^{2}\right] .
\end{align*}
Moreover, since $\left\| \Phi (Z_{n}) \right\| \leq C \left\| Z_{n} - m \right\|$, applying Theorem 4.1, we have,
\begin{align*}
\frac{1}{n^{2}}\sum_{k=1}^{n} \mathbb{E}\left[ \left\| \Phi (Z_{k} ) \right\|^{2} \right] & \leq \frac{C^{2}}{n^{2}}\sum_{k=1}^{n}\mathbb{E}\left[ \left\| Z_{k} - m \right\|^{2}\right] \\
& \leq \frac{C^{2}K_{1}}{n^{2}}\sum_{k=1}^{n}\frac{1}{k^{\alpha}}\\
& = o \left( \frac{1}{n}\right) .
\end{align*}
Finally,
\begin{equation}
\mathbb{E}\left[ \left\| \Gamma_{m}\left( \overline{Z}_{n} - m \right) \right\|^{2}\right] = \frac{1}{n} + o \left(  \frac{1}{n}\right) .
\end{equation}
Thus, since the largest eigenvalue of $\Gamma_{m}$ satisfies $\lambda_{\max} \leq C$, there is a rank $n_{\alpha}$ such that for all $n \geq n_{\alpha}$,
\[
\mathbb{E}\left[ \left\| \overline{Z}_{n} -m\right\|^{2}\right] \geq \frac{1}{2C^{2}n}.
\]
Let $c'' := \min \left\lbrace \min_{1 \leq k\leq n_{\alpha}}\left\lbrace k\mathbb{E}\left[ \left\| \overline{Z}_{k} - m \right\|^{2}\right] \right\rbrace , \frac{1}{2C^{2}} \right\rbrace$, for all $n \geq 1$,
\begin{equation}
\mathbb{E}\left[ \left\| \overline{Z}_{n} - m \right\|^{2} \right] \geq \frac{c''}{n}.
\end{equation}

\end{proof}

\subsection{Proofs of Section 5}
\begin{proof}[Proof of Theorem 5.1] Let $\beta ' \in (1/2,1)$ such that $\beta ' < \alpha $. In order to apply Borel-Cantelli's Lemma, we will prove that
\begin{equation}
\sum_{n\geq 1} \mathbb{P}\left( \left\| Z_{n} - m \right\| \geq \frac{1}{n^{\beta ' /2}} \right) < \infty .
\end{equation}
Applying Theorem 4.1, for all $p \geq 1$, for all $n\geq 1$,
\begin{align*}
\mathbb{P}\left( \left\| Z_{n} - m \right\| \geq \frac{1}{n^{\beta ' /2}} \right) & \leq \mathbb{E}\left[ \left\| Z_{n} - m \right\|^{2p}\right]n^{p\beta '} \\
& \leq \frac{K_{p}}{n^{p ( \alpha - \beta ' )}} .
\end{align*}
Since $\beta ' < \alpha $, we can take $p > \frac{1}{\alpha - \beta '}$ and we get
\begin{align*}
\sum_{n\geq 1} \mathbb{P}\left( \left\| Z_{n} - m \right\| \geq \frac{1}{n^{\beta ' /2}} \right) \leq  \sum_{n\geq 1} \frac{K_{p}}{n^{p( \alpha - \beta ' ) }} < \infty .
\end{align*}
Applying Borel-Cantelli's Lemma, 
\begin{equation}
\left\| Z_{n} - m \right\| = O \left( n^{-\frac{\beta '}{2}} \right) \quad a.s,
\end{equation}
for all $\beta ' < \alpha$. In a particular case, for all $\beta < \alpha$,
\begin{equation}
\left\| Z_{n} - m \right\| = o \left( n^{-\frac{\beta}{2}}\right) \quad a.s.
\end{equation}
\end{proof}
\begin{proof}[Proof of Corollary 5.1]
Let us recall decomposition (\ref{decmoy}) of the averaged algorithm:
\begin{equation}\notag
\Gamma_{m}\left( \overline{Z}_{n} - m \right) = \frac{1}{n}\left( \frac{T_{1}}{\gamma_{1}}- \frac{T_{n+1}}{\gamma_{n}}+ \sum_{k=2}^{n}T_{k}\left( \frac{1}{\gamma_{k}}- \frac{1}{\gamma_{k-1}}\right) + \sum_{k=1}^{n} \delta_{k} + \sum_{k=1}^{n}\xi_{k+1} \right)  .
\end{equation}
We will give the almost sure rate of convergence of each term. First, since $Z_{1}$ is bounded, we have $\left\| \frac{T_{1}}{n\gamma_{1}} \right\| = O \left( \frac{1}{n}\right) $ almost surely. Applying Theorem 5.1, let $\beta' < \alpha$, 
\begin{align*}
\left\| \frac{T_{n+1}}{n\gamma_{n}} \right\| &  = o \left( \frac{n^{-\frac{\beta '}{2}}}{n^{1-\alpha}} \right)\quad a.s \\
& = o \left( \frac{1}{\sqrt{n}}\right) \quad a.s.
\end{align*}
Indeed, we obtain the last equality by taking $ \alpha > \beta ' > 2\alpha -1$, which is possible since $\alpha <1$. Moreover, since $\left| \gamma_{k}^{-1} - \gamma_{k-1}^{-1}\right|\leq 2\alpha c_{\gamma}^{-1}k^{\alpha -1}$, let $\beta ' < \alpha$, applying Theorem 5.1, 
\begin{align*}
\left\| \frac{1}{n}\sum_{k=2}^{n} T_{k} \left( \frac{1}{\gamma_{k}} - \frac{1}{\gamma_{k-1}}\right) \right\| & \leq \frac{1}{n}\sum_{k=1}^{n} \left\| T_{k} \right\| \left| \frac{1}{\gamma_{k-1}} - \frac{1}{\gamma_{k}}\right| \\
 & = o \left( \frac{1}{n}\sum_{k=2}^{n}k^{\alpha - \beta ' /2 -1} \right) \quad a.s \\
& = o \left( \frac{n^{\alpha - \beta ' /2 }}{n} \right) \quad a.s \\
& = o \left( \frac{1}{\sqrt{n}}\right) \quad a.s .
\end{align*}
Indeed, we get the last equality taking $\beta' > 2\alpha -1$. Moreover, since $\left\| \delta_{n} \right\| \leq C_{m} \left\| Z_{n} - m \right\|^{2}$, for all $\beta ' < \alpha $, 
\begin{align*}
\left\| \frac{1}{n}\sum_{k=1}^{n}\delta_{k}\right\|&  \leq \frac{1}{n}\sum_{k=1}^{n} \left\| \delta_{k} \right\| \\
& \leq  \frac{C_{m}}{n} \sum_{k=1}^{n}\left\| Z_{k} - m \right\|^{2} \\
& = o \left( \frac{1}{n}\sum_{k=1}^{n}\frac{1}{k^{\beta '}} \right) \quad a.s \\
& = o \left( \frac{1}{n^{\beta '}} \right) \quad a.s  \\
& = o \left( \frac{1}{\sqrt{n}}\right) \quad a.s.
\end{align*}
Indeed, we obtain the last equality by taking $\alpha > \beta ' > 1/2$. Finally, since $\mathbb{E}\left[ \left\| \sum_{k=1}^{n}\xi_{k+1}\right\|^{2}\right]~ = ~n ~+ ~o \left( n \right)$ (see \cite{CCG2015} and proof of Theorem 4.2), applying the law of large numbers for martingales (see Theorem 1.3.15 in \cite{Duf97}), for all $\delta > 0$, 
\begin{equation}\label{lgn}
\frac{1}{n}\sum_{k=1}^{n}\xi_{n+1} = o\left( \frac{\left(\ln n \right)^{\frac{1+\delta}{2}}}{\sqrt{n}}\right) \quad a.s ,
\end{equation}
which concludes the proof.
\end{proof}
\begin{rmq}
Note that the law of large numbers for martingales in \cite{Duf97} is not given for general Hilbert spaces. Nevertheless, in our context, this law of large numbers can be extended. We just have to prove that for all positive constant $\delta$, $U_{n}:= \frac{1}{\sqrt{n (\ln (n))^{1+\delta}}}\left\| \sum_{k=1}^{n}\xi_{k+1}\right\|$ converges almost surely to a finite random variable. Since $\left( \xi_{n}\right)$ is a sequence of martingale differences adapted to the filtration $\left( \mathcal{F}_{n} \right)$, and since $\mathbb{E}\left[ \left\| \xi_{n+1} \right\|^{2}\big| \mathcal{F}_{n} \right] \leq 1$,
\begin{align*}
\mathbb{E}\left[  U_{n+1}^{2} \big| \mathcal{F}_{n} \right] & = \frac{n (\ln (n))^{1+\delta}}{(n+1) (\ln (n+1))^{1+\delta}}U_{n}^{2} + \frac{1}{(n+1) (\ln (n+1))^{1+\delta}}\mathbb{E}\left[ \left\| \xi_{n+1} \right\|^{2}\big| \mathcal{F}_{n} \right] \\
& \leq U_{n}^{2} + \frac{1}{(n+1) (\ln (n+1))^{1+\delta}}.
\end{align*} 
Thus, applying Robbins-Siegmund Theorem (see \cite{Duf97}), $\left( U_{n}\right)$ converges almost surely to a finite random variable, which concludes the proof.
\end{rmq}

\end{appendix}
\bibliographystyle{apalike}
\bibliography{biblio_redaction}

\def\cprime{$'$}
\begin{thebibliography}{}

\bibitem[{A}rnaudon et~al., 2012]{ADPY10}
{A}rnaudon, M., {D}ombry, C., {P}han, A., and {Y}ang, L. (2012).
\newblock {S}tochastic algorithms for computing means of probability measures.
\newblock {\em Stochastic Processes and their Applications}, 122:1437--1455.

\bibitem[Bali et~al., 2011]{BBTW2011}
Bali, J.~L., Boente, G., Tyler, D.~E., and Wang, J.-L. (2011).
\newblock Robust functional principal components: a projection-pursuit
  approach.
\newblock {\em The Annals of Statistics}, 39(6):2852--2882.

\bibitem[Bartoli and Del~Moral, 2001]{bartoli}
Bartoli, N. and Del~Moral, P. (2001).
\newblock Simulation et algorithmes stochastiques.
\newblock {\em C{\'e}padu{\`e}s {\'e}ditions}.

\bibitem[Beck and Sabach, 2014]{beck2013weiszfeld}
Beck, A. and Sabach, S. (2014).
\newblock Weiszfeld's method: Old and new results.
\newblock {\em Journal of Optimization Theory and Applications}, to appear.

\bibitem[Benveniste et~al., 1990]{benveniste-book90}
Benveniste, A., M{\'e}tivier, M., and Priouret, P. (1990).
\newblock {\em Adaptive Algorithms and Stochastic Approximations}, volume~22 of
  {\em Applications of Mathematics}.
\newblock Springer-Verlag, New York.

\bibitem[Bongiorno et~al., 2014]{bongiorno2014contributions}
Bongiorno, E.~G., Salinelli, E., Goia, A., and Vieu, P. (2014).
\newblock {\em Contributions in infinite-dimensional statistics and related
  topics}.
\newblock Societ{\`a} Editrice Esculapio.

\bibitem[Cadre, 2001]{Cad01}
Cadre, B. (2001).
\newblock Convergent estimators for the {$L\sb 1$}-median of a {B}anach valued
  random variable.
\newblock {\em Statistics}, 35(4):509--521.

\bibitem[Cardot et~al., 2015]{CCG2015}
Cardot, H., C\'enac, P., and Godichon, A. (2015).
\newblock Online estimation of the geometric median in hilbert spaces: non
  asymptotic confidence balls.
\newblock Technical report, arXiv:1501.06930.

\bibitem[Cardot et~al., 2012]{CCM10}
Cardot, H., C{\'e}nac, P., and Monnez, J.-M. (2012).
\newblock A fast and recursive algorithm for clustering large datasets with
  k-medians.
\newblock {\em Computational Statistics \& Data Analysis}, 56(6):1434--1449.

\bibitem[Cardot et~al., 2013]{HC}
Cardot, H., C{\'e}nac, P., and Zitt, P.-A. (2013).
\newblock Efficient and fast estimation of the geometric median in {H}ilbert
  spaces with an averaged stochastic gradient algorithm.
\newblock {\em Bernoulli}, 19(1):18--43.

\bibitem[Chakraborty and Chaudhuri, 2014]{chakraborty2014spatial}
Chakraborty, A. and Chaudhuri, P. (2014).
\newblock The spatial distribution in infinite dimensional spaces and related
  quantiles and depths.
\newblock {\em The Annals of Statistics}, 42:1203--1231.

\bibitem[Chaudhuri, 1992]{Chaud92}
Chaudhuri, P. (1992).
\newblock Multivariate location estimation using extension of {$R$}-estimates
  through {$U$}-statistics type approach.
\newblock {\em Ann. Statist.}, 20:897--916.

\bibitem[Cuevas, 2014]{cuevas2014partial}
Cuevas, A. (2014).
\newblock A partial overview of the theory of statistics with functional data.
\newblock {\em Journal of Statistical Planning and Inference}, 147:1--23.

\bibitem[Duflo, 1997]{Duf97}
Duflo, M. (1997).
\newblock {\em Random iterative models}, volume~34 of {\em Applications of
  Mathematics (New York)}.
\newblock Springer-Verlag, Berlin.
\newblock Translated from the 1990 French original by Stephen S. Wilson and
  revised by the author.

\bibitem[Ferraty and Vieu, 2006]{ferraty2006nonparametric}
Ferraty, F. and Vieu, P. (2006).
\newblock {\em Nonparametric functional data analysis: theory and practice}.
\newblock Springer Science \& Business Media.

\bibitem[Gervini, 2008]{Ger08}
Gervini, D. (2008).
\newblock Robust functional estimation using the median and spherical principal
  components.
\newblock {\em Biometrika}, 95(3):587--600.

\bibitem[Haldane, 1948]{Hal48}
Haldane, J. B.~S. (1948).
\newblock {Note on the median of a multivariate distribution}.
\newblock {\em Biometrika}, 35(3-4):414--417.

\bibitem[Kemperman, 1987]{Kem87}
Kemperman, J. H.~B. (1987).
\newblock The median of a finite measure on a {B}anach space.
\newblock In {\em Statistical data analysis based on the {$L\sb 1$}-norm and
  related methods ({N}euch\^atel, 1987)}, pages 217--230. North-Holland,
  Amsterdam.

\bibitem[Kuhn, 1973]{kuhn1973note}
Kuhn, H.~W. (1973).
\newblock A note on {F}ermat's problem.
\newblock {\em Mathematical programming}, 4(1):98--107.

\bibitem[Kushner and Yin, 2003]{kushner2003stochastic}
Kushner, H.~J. and Yin, G. (2003).
\newblock {\em Stochastic approximation and recursive algorithms and
  applications}, volume~35.
\newblock Springer.

\bibitem[Minsker, 2014]{minsker2013geometric}
Minsker, S. (2014).
\newblock Geometric median and robust estimation in {B}anach spaces.
\newblock {\em Bernoulli, to appear}.

\bibitem[M\"ott\"onen et~al., 2010]{MNO2010}
M\"ott\"onen, J., Nordhausen, K., and Oja, H. (2010).
\newblock Asymptotic theory of the spatial median.
\newblock In {\em Nonparametrics and Robustness in Modern Statistical Inference
  and Time Series Analysis: A Festschrift in honor of Professor Jana
  Jure{\u{c}}kov{\'{a}}}, volume~7, pages 182--193. IMS Collection.

\bibitem[Pelletier, 2000]{Pel00}
Pelletier, M. (2000).
\newblock Asymptotic almost sure efficiency of averaged stochastic algorithms.
\newblock {\em SIAM J. Control Optim.}, 39(1):49--72.

\bibitem[Petrov, 1995]{petrov1995limit}
Petrov, V.~V. (1995).
\newblock Limit theorems of probability theory. sequences of independent random
  variables, vol. 4 of.
\newblock {\em Oxford Studies in Probability}.

\bibitem[Polyak and Juditsky, 1992]{PolyakJud92}
Polyak, B. and Juditsky, A. (1992).
\newblock Acceleration of stochastic approximation.
\newblock {\em SIAM J. Control and Optimization}, 30:838--855.

\bibitem[Robbins and Monro, 1951]{robbins1951stochastic}
Robbins, H. and Monro, S. (1951).
\newblock A stochastic approximation method.
\newblock {\em The annals of mathematical statistics}, pages 400--407.

\bibitem[Silverman and Ramsay, 2005]{silverman2005functional}
Silverman, B. and Ramsay, J. (2005).
\newblock {\em Functional Data Analysis}.
\newblock Springer.

\bibitem[Smarandache, 1996]{smarandache1996collected}
Smarandache, F. (1996).
\newblock {\em Collected Papers, Vol. I}, volume~1.
\newblock Infinite Study.

\bibitem[Vardi and Zhang, 2000]{VZ00}
Vardi, Y. and Zhang, C.-H. (2000).
\newblock The multivariate {$L\sb 1$}-median and associated data depth.
\newblock {\em Proc. Natl. Acad. Sci. USA}, 97(4):1423--1426.

\bibitem[Weber, 1929]{weber1929alfred}
Weber, A. (1929).
\newblock Uber den standort der industrien (alfred weber's theory of the
  location of industries).
\newblock {\em University of Chicago}.

\bibitem[Yang, 2010]{yang2010riemannian}
Yang, L. (2010).
\newblock Riemannian median and its estimation.
\newblock {\em LMS Journal of Computation and Mathematics}, 13:461--479.

\end{thebibliography}

\end{document}